\documentclass[final,leqno,onefignum,onetabnum]{siamltexmm}
\usepackage{amsmath}
\usepackage{amsbsy}
\usepackage{amssymb}
\usepackage[dvips]{graphicx}
\usepackage{arydshln}
\newtheorem{assumption}[theorem]{Assumption}

\usepackage{geometry}
\usepackage{enumerate}
\usepackage{autonum}
\usepackage{scalerel}
\usepackage{threeparttable}
\DeclareMathOperator*{\bigtimes}{\scalerel*{\times}{\sum}}
\newgeometry{left=0.75in,right=0.75in,top=1in,bottom=0.7in}

\title{A Spline Dimensional Decomposition for Uncertainty Quantification in High Dimensions
\thanks{This work was supported by the U.S. National Science Foundation under Grant Nos. CMMI-1607398 and CMMI-1933114.}}

\author{Sharif Rahman \and Ramin Jahanbin
\thanks{College of Engineering and Applied Mathematics \& Computational Sciences, The University of Iowa, Iowa City, IA 52242 (\email{sharif-rahman@uiowa.edu}). Questions, comments, or corrections
to this document may be directed to that email address.}}

\begin{document}

\maketitle
\newcommand{\slugmaster}{%
\slugger{juq}{xxxx}{xx}{x}{x--x}}

\begin{abstract}
This study debuts a new spline dimensional decomposition (SDD) for uncertainty quantification analysis of high-dimensional functions, including those endowed with high nonlinearity and nonsmoothness, if they exist, in a proficient manner. The decomposition creates an hierarchical expansion for an output random variable of interest with respect to measure-consistent orthonormalized basis splines (B-splines) in independent input random variables.  A dimensionwise decomposition of a spline space into orthogonal subspaces, each spanned by a reduced set of such orthonormal splines, results in SDD.  Exploiting the modulus of smoothness, the SDD approximation is shown to converge in mean-square to the correct limit.  The computational complexity of the SDD method is polynomial, as opposed to exponential, thus alleviating the curse of dimensionality to the extent possible. Analytical formulae are proposed to calculate the second-moment properties of a truncated SDD approximation for a general output random variable in terms of the expansion coefficients involved. Numerical results indicate that a low-order SDD approximation of nonsmooth functions calculates the probabilistic characteristics of an output variable with an accuracy matching or surpassing those obtained by high-order approximations from several existing methods. Finally, a 34-dimensional random eigenvalue analysis demonstrates the utility of SDD in solving practical problems.
\end{abstract}


\pagestyle{myheadings}
\thispagestyle{plain}
\markboth{S. RAHMAN}{A SPLINE DIMENSIONAL DECOMPOSITION}

\section{Introduction}
Uncertainty quantification, commonly referred to as UQ, is concerned with characterizing, propagating, and managing statistical variability in computer models of complex systems \cite{smith13,sullivan15}.  It is usually performed in conjunction with stochastic methods that estimate the statistical moments, probability law, and other relevant properties of an output variable of interest.  A family of popular methods, namely, the polynomial chaos expansion (PCE) \cite{cameron47,wiener38} and polynomial dimensional decomposition (PDD) methods \cite{rahman18}, is premised on the smoothness property of the output function because the polynomial basis of the expansion or decomposition is also smooth and globally supported.  Although these polynomial-based methods have played a central role in UQ for many years, their main drawback is inflexibility. While PCE and PDD seem to do all right with sufficiently low expansion orders or degrees
\footnote{The nouns \emph{degree} and \emph{order} are used synonymously in this paper when referring to a polynomial or spline expansion.}
meant for smooth functions, they are inflicted with uncontrolled fluctuations if the expansion order is excessively high, causing unreliable predictions of stochastic performance for nonsmooth functions. This observation suggests that in order to attain better approximation quality for nonsmooth functions on a large domain, one should work with smoothly connected piecewise polynomials, \emph{a.k.a} splines, of relatively low expansion orders and smaller subdomains.

Recently, a spline chaos expansion (SCE), comprising orthonormalized basis splines (B-splines) in input random variables, has been developed to tackle locally pronounced, highly nonlinear, or nonsmooth functions \cite{rahman20}.  The expansion is similar to PCE, but, by swapping polynomials for B-splines, SCE achieves a greater flexibility in selecting expansion orders and dealing with subdomains.  In consequence, a low-degree SCE approximation with an adequate mesh size produces a markedly more accurate estimate of the output variance of a nonsmooth function than the commonly used PCE with overly large expansion orders.  However, due to the tensor-product structure, SCE, like PCE, suffers from the curse of dimensionality.  The curse becomes worse for SCE as there are substantially more basis functions in SCE than in PCE.  Since PDD, equipped with a desirable dimensional hierarchy of input variables, is able to deflate PCE's curse of dimensionality to the extent possible, there is a cautious optimism that the same feat can be repeated in the context of SCE.  This is the chief motivation for this work.

The B-splines have been exploited to build the sparse-grid methods \cite{gerstner98,smolyak63}, but they are not orthogonal with respect to the probability measure of input random variables.  More often than not, the resultant approximations are only $C^0$-continuous at element boundaries, as is also the case for the multi-element collocation method \cite{foo08}. More notably, these methods are entrenched in the referential dimensional decomposition (RDD) \cite{rahman14}, also known as anchored decomposition \cite{griebel10} or cut-HDMR \cite{rabitz99}, of a high-dimensional function.  In contrast, the method developed in this work is founded on the analysis-of-variance (ANOVA) dimensional decomposition (ADD), which is generally superior to RDD. Indeed, an error analysis reveals sub-optimality of RDD approximations, indicating that an RDD approximation, regardless of how the reference point is selected, cannot be better than an ADD approximation for the same degrees of interaction \cite{rahman14}.

This paper presents a new, alternative dimensionwise orthogonal expansion, referred to as spline dimensional decomposition or SDD, for  a general UQ analysis of high-dimensional functions, including those featuring nonsmoothness, subject to independent but otherwise arbitrary probability measures of input random variables. It focuses on the fundamentals and mathematical aspects of SDD, followed by a few applications. Readers interested in further applications to isogeometric analysis and latter developments of an optimal version should consult the companion papers \cite{dixler21,jahanbin20}. The paper is organized as follows.  Section 2 starts with mathematical preliminaries and assumptions.  A brief account of univariate B-splines, including their orthonormalized version, is presented in Section 3. Section 4 explains the construction of dimensionwise multivariate B-splines for generating an orthonormal basis of a spline space of interest.  Section 5 properly introduces SDD for a square-integrable random variable and then affirms the convergence and optimality of SDD.  The formulae for the mean and variance of an SDD approximation are also deduced.  Numerical results from two example problems are discussed in Section 6. Section 7 demonstrates the power of the SDD method by solving a large-scale engineering problem from the automotive industry. Finally, the conclusions are drawn in Section 8.

\section{Input random variables}
Let $\mathbb{N}:=\{1,2,\ldots\}$, $\mathbb{N}_{0}:=\mathbb{N} \cup \{0\}$, and $\mathbb{R}:=(-\infty,+\infty)$ represent the sets of positive integer (natural), non-negative integer, and real numbers, respectively.  Denote by $[a_k,b_k]$ a finite closed interval, where $a_k, b_k \in \mathbb{R}$, $b_k>a_k$. Then, given $N \in \mathbb{N}$, $\mathbb{A}^N=\times_{k=1}^N [a_k,b_k]$ represents a closed bounded domain of $\mathbb{R}^N$.

Let $(\Omega,\mathcal{F},\mathbb{P})$ be a probability space, where $\Omega$ is a sample space representing an abstract set of elementary events, $\mathcal{F}$ is a $\sigma$-algebra on $\Omega$, and $\mathbb{P}:\mathcal{F}\to[0,1]$ is a probability measure.  Defined on this probability space, consider an $N$-dimensional input random vector $\mathbf{X}:=(X_{1},\ldots,X_{N})^\intercal$, describing the statistical uncertainties in all system parameters of a stochastic or UQ problem.  Denote by $F_{\mathbf{X}}({\mathbf{x}}):=\mathbb{P}(\cap_{i=1}^{N}\{ X_k \le x_k \})$ the joint distribution function of $\mathbf{X}$. The $k$th component of $\mathbf{X}$ is a random variable $X_k$, which has the marginal probability distribution function $F_{X_k}(x_k):=\mathbb{P}(X_k \le x_k)$.  The non-zero, finite integer $N$ represents the number of input random variables and is often referred to as the dimension of the stochastic or UQ problem. Albeit subjective, here, the stochastic dimension exceeding ten is considered to be a high-dimensional UQ problem.

A set of assumptions on input random variables used or required by SDD is as follows.

\begin{assumption}
The input random vector $\mathbf{X}:=(X_{1},\ldots,X_{N})^\intercal$ satisfies all of the following conditions:
\begin{enumerate}[{$(1)$}]

\item
All component random variables $X_k$, $k=1,\ldots,N$, are statistically independent, but not necessarily identically distributed.

\item
Each input random variable $X_k$ is defined on a bounded interval $[a_k,b_k] \subset \mathbb{R}$.  Therefore, all moments of $X_k$ exists, that is, for all $l \in \mathbb{N}_0$,
\begin{equation}
\mathbb{E} \left[ X_k^l \right] :=
\int_{\Omega} X_k^l(\omega) d\mathbb{P}(\omega) < \infty,
\label{2.1}
\end{equation}
where $\mathbb{E}$ is the expectation operator with respect to the probability measure $\mathbb{P}$.

\item
Each input random variable $X_k$ has absolutely continuous marginal probability distribution function $F_{X_k}(x_k)$ and continuous marginal probability density function $f_{X_k}(x_k):={\partial F_{X_k}(x_k)}/{\partial x_k}$ with a bounded support $[a_k,b_k] \subset \mathbb{R}$.  Consequently, with Items (1) and (2) in mind, the joint probability distribution function $F_{\mathbf{X}}({\mathbf{x}})$ and joint probability density function $f_{\mathbf{X}}({\mathbf{x}}):={\partial^N F_{\mathbf{X}}({\mathbf{x}})}/{\partial x_1 \cdots \partial x_N}$ of $\mathbf{X}$ are obtained from
\[
F_{\mathbf{X}}({\mathbf{x}})=\prod_{k=1}^{N} F_{X_k}(x_k)~~\text{and}~~
f_{\mathbf{X}}({\mathbf{x}})=\prod_{k=1}^{N} f_{X_k}(x_k),
\]
respectively, with a bounded support $\mathbb{A}^N \subset \mathbb{R}^N$ of the density function.

\end{enumerate}
\label{a1}
\end{assumption}

Assumption \ref{a1} warrants the existence of a relevant sequence of orthogonal polynomials or splines consistent with the input probability measure. The discrete distributions and dependent variables are not considered in the paper.

Given the abstract probability space $(\Omega,\mathcal{F},\mathbb{P})$ of  $\mathbf{X}$, there is an image probability space $(\mathbb{A}^N,\mathcal{B}^{N},f_{\mathbf{X}}d\mathbf{x})$, where $\mathbb{A}^N$ is the image of $\Omega$ from the mapping $\mathbf{X}:\Omega \to \mathbb{A}^N$ and $\mathcal{B}^{N}:=\mathcal{B}(\mathbb{A}^{N})$ is the Borel $\sigma$-algebra on $\mathbb{A}^N \subset \mathbb{R}^N$.  Appropriate statements and objects in the abstract probability space have obvious counterparts in the associated image probability space. Both probability spaces will be used in the paper.

\section{Univariate Orthonormal Splines}
Let $\mathbf{x}=(x_1,\ldots,x_N)$ be an arbitrary point in $\mathbb{A}^N$.  For the coordinate direction $k$, $k=1,\ldots,N$, define a non-negative integer $p_k \in \mathbb{N}_0$ and a positive integer $n_k \ge (p_k+1) \in \mathbb{N}$, representing the degree and total number of basis functions, respectively. Then, a knot sequence
\begin{equation}
\begin{array}{c}
\boldsymbol{\xi}_k:={\{\xi_{k,i_k}\}}_{i_k=1}^{n_k+p_k+1}=
\{a_k=\xi_{k,1},\xi_{k,2},\ldots,\xi_{k,n_k+p_k+1}=b_k\},  \\
\xi_{k,1} \le \xi_{k,2} \le \cdots \le \xi_{k,n_k+p_k+1},
\end{array}
\label{3.1}
\end{equation}
is defined on the interval $[a_k,b_k] \subset \mathbb{R}$ by a non-decreasing sequence of real numbers, where  $\xi_{k,i_k}$ is the $i_k$th knot with $i_k=1,2,\ldots,n_k+p_k+1$ representing the knot index for the coordinate direction $k$. A knot $\xi_{k,i_k}$ may appear up to $p_k+1$ times. Therefore, the knot sequence can also be written as
\begin{equation}
\begin{array}{c}
\boldsymbol{\xi}_k=\{a_k=\overset{m_{k,1}~\mathrm{times}}{\overbrace{\zeta_{k,1},\ldots,\zeta_{k,1}}},
\overset{m_{k,2}~\mathrm{times}}{\overbrace{\zeta_{k,2},\ldots,\zeta_{k,2}}},\ldots,
\overset{m_{k,r_{k}-1}~\mathrm{times}}{\overbrace{\zeta_{k,r_{k}-1},\ldots,\zeta_{k,r_{k}-1}}},
\overset{m_{k,r_{k}}~\mathrm{times}}{\overbrace{\zeta_{k,r_{k}},\ldots,\zeta_{k,r_{k}}}}=b_k\}, \\
a_k=\zeta_{k,1} < \zeta_{k,2} < \cdots < \zeta_{k,r_k-1} < \zeta_{k,r_k}=b_k,  \rule{0pt}{0.2in}
\end{array}
\label{3.2}
\end{equation}
where $\zeta_{k,j_k}$, $j_k=1,\ldots,r_k$, are $r_k$ unique knots, each of which has multiplicity $1 \le m_{k,j_k} \le p_k+1$. A knot sequence is said to be $(p_k+1)$-open if the first and last knots appear $p_k+1$ times. Furthermore, a knot sequence is said to be $(p_k+1)$-open with simple knots if it is $(p_k+1)$-open and all interior knots appear only once. A $(p_k+1)$-open knot sequence with or without simple knots is commonly found in applications \cite{cottrell09}. For further details, read Appendix A of this paper or Chapter 2 of the book by Cottrell \emph{et al.} \cite{cottrell09}.

\subsection{Standard B-splines}
Let $\boldsymbol{\xi}_k$ be a general knot sequence of length at least $p_k+2$ for the interval $[a_k,b_k]$,
as defined by \eqref{3.1}. Denote by $B_{i_k,p_k,\boldsymbol{\xi}_k}^k(x_k)$ the $i_k$th univariate B-spline function with degree $p_k \in \mathbb{N}_0$ for the coordinate direction $k$. Given the zero-degree basis functions,
\begin{equation}
B_{i_k,0,\boldsymbol{\xi}_k}^k(x_k) :=
\begin{cases}
1, & \xi_{k,i_k} \le x_k < \xi_{k,i_k+1}, \\
0, & \text{otherwise},
\end{cases}
\label{3.5}
\end{equation}
for $k=1,\ldots,N$, all higher-order B-spline functions on $\mathbb{R}$ are defined recursively by \cite{cox72,deboor72,piegl97}
\begin{equation}
B_{i_k,p_k,\boldsymbol{\xi}_k}^k(x_k) :=
\frac{x_k - \xi_{k,i_k}}{\xi_{k,i_k+p_k} - \xi_{k,i_k}} B_{i_k,p_k-1,\boldsymbol{\xi}_k}^k(x_k) +
\frac{\xi_{k,i_k+p_k+1}-x_k}{\xi_{k,i_k+p_k+1}-\xi_{k,i_k+1}} B_{i_k+1,p_k-1,\boldsymbol{\xi}_k}^k(x_k),
\label{3.6}
\end{equation}
where $1 \le k \le N$, $1 \le i_k \le n_k$, $1 \le p_k < \infty$, and $0/0$ is considered as \emph{zero}.

The B-spline functions are endowed with a number of desirable properties \cite{cox72,deboor72,piegl97}.  They are non-negative, locally supported on a subinterval, linearly independent, and committed to partition of unity \cite{cottrell09}. A B-spline is also everywhere pointwise $C^\infty$-continuous except at the knots $\zeta_{k,i_k}$ of multiplicity $m_{k,i_k}$, where it is $C^{p_k-m_{k,i_k}}$-continuous, provided that $1 \le m_{k,i_k} < p_k+1$.  Some of these properties will be exploited in the development of SDD.

\subsection{Spline space}
Suppose for $n_k > p_k \ge 0$, a knot sequence $\boldsymbol{\xi}_k$ has been specified on the interval $[a_k,b_k]$.  The associated spline space of degree $p_k$, denoted by $\mathcal{S}_{k,p_k,\boldsymbol{\xi}_k}$, is conveniently defined using an appropriate polynomial space. Define such a polynomial space as a finite-dimensional linear space
\[
\Pi_{p_k}:=
\left\{
g(x_k) =
\displaystyle
\sum_{l=0}^{p_k} c_{k,l} x_k^l: c_{k,l} \in \mathbb{R}
\right\}
\]
of real-valued polynomials in $x_k$ of degree at most $p_k$.

\begin{definition}[Schumaker \cite{schumaker07}]
For $n_k > p_k \ge 0$, let $\boldsymbol{\xi}_k$ be a
$(p_k+1)$-open knot sequence on the interval $[a_k,b_k]$, as defined by \eqref{3.2} with $m_{k,1}=m_{k,r_k}=p_k+1$. Then the space
\begin{equation}
\mathcal{S}_{k,p_k,\boldsymbol{\xi}_k} :=
\left \{
\begin{array}{l}
g_k:[a_k,b_k] \to \mathbb{R}: ~\text{there exist polynomials}~g_{k,1},g_{k,2},\ldots,g_{k,r_k-1} ~\text{in}~ \Pi_{p_k}
\\
\text{such that}~g_k(x_k)=g_{k,i_k}(x_k)~\text{for}~x_k \in [\xi_{k,i_k},\xi_{k,i_k+1}),~i_k=1,\ldots,r_k-1, \\
\text{and}~
\displaystyle
\frac{\partial^{j_k} g_{k,i_k-1}}{\partial x_k}(\xi_{k,i_k}) =
\displaystyle
\frac{\partial^{j_k} g_{k,i_k}}{\partial x_k}(\xi_{k,i_k})~\text{for}~j_k=0,1,\ldots,p_k-m_{k,i_k},
\\
i_k=2,\ldots,r_k-1
\end{array}
\right\}
\label{3.7}
\end{equation}
is defined as the spline space of degree $p_k$ with distinct knots $\zeta_{k,1},\ldots,\zeta_{k,r_k}$ of multiplicities $m_{k,1}=p_k+1$, $1 \le m_{k,2} \le p_k+1$, $\ldots$, $1 \le m_{k,r_k-1} \le p_k+1$, $m_{k,r_k}=p_k+1$.
\label{d5}
\end{definition}

The spline space is uniquely described by distinct interior knots $\zeta_{k,2},\ldots,\zeta_{k,r_k-1}$ of multiplicities $m_{k,2},\ldots,m_{k,r_k-1}$.  Indeed, the multiplicities determine the structure of $\mathcal{S}_{k,p_k,\boldsymbol{\xi}_k}$ by regulating the smoothness of the splines at interior knots.  For instance, if $m_{k,i_k}=p_k+1$, $i_k=2,\ldots,r_k-1$, then two polynomial pieces $g_{k,i_k-1}$ and $g_{k,i_k}$ in the sub-intervals associated with the knot $\xi_{k,i_k}$ are unattached, possibly creating a jump discontinuity at $\xi_{k,i_k}$.  In this case, $\mathcal{S}_{k,p_k,\boldsymbol{\xi}_k}$ will be the roughest space of splines.  If $m_{k,i_k}<p_k+1$, $i_k=2,\ldots,r_k-1$, then the two aforementioned polynomial pieces are joined smoothly in the sense that the first $p_k-m_{k,i_k}$ derivatives are all continuous across the knot. More specifically, if $m_{k,i_k}=1$, $i_k=2,\ldots,r_k-1$, then there are simple knots with the corresponding spline space becoming the smoothest space of piecewise polynomials of degree at most $p_k$.

\begin{proposition}
The spline space $\mathcal{S}_{k,p_k,\boldsymbol{\xi}_k}$ is a linear space of dimension
\begin{equation}
\dim \mathcal{S}_{k,p_k,\boldsymbol{\xi}_k} = n_k = \displaystyle \sum_{i_k=2}^{r_k-1} m_{k,i_k} + p_k +1.
\label{3.8}
\end{equation}
\label{p1}
\end{proposition}

\begin{proposition}
For $n_k > p_k \ge 0$, let $\boldsymbol{\xi}_k$ be a $(p_k+1)$-open knot sequence on the interval $[a_k,b_k]$. Denote by
\begin{equation}
\left\{
B_{1,p_k,\boldsymbol{\xi}_k}^k(x_k),\ldots,B_{n_k,p_k,\boldsymbol{\xi}_k}^k(x_k)
\right\}
\label{3.9}
\end{equation}
a set of $n_k$ B-splines of degree $p_k$. Then
\begin{equation}
\mathcal{S}_{k,p_k,\boldsymbol{\xi}_k} :=
\operatorname{span}\{B_{i_k,p_k,\boldsymbol{\xi}_k}^k(x_k)\}_{i_k=1,\ldots,n_k}.
\label{3.10}
\end{equation}
\label{p2}
\end{proposition}

\subsection{Orthonormalized B-splines}
The B-splines presented in the previous subsection, although they form a basis of the spline space $\mathcal{S}_{k,p_k,\boldsymbol{\xi}_k}$, are not orthogonal with respect to the probability measure $f_{X_k}(x_k)dx_k$.  A popular choice for building an orthogonal or orthonormal basis is the well-known Gram-Schmidt procedure \cite{golub96}. However, it is known to be ill-conditioned.  In this section, a linear transformation, originally proposed during the development of SCE \cite{rahman20}, is briefly summarized here in three steps to generate their orthonormal version, as follows.

\vspace{0.1in}
\begin{enumerate}[{$(1)$}]

\item
Given the set of B-splines in \eqref{3.9}, replace any one of its elements with an arbitrary non-zero constant, thus creating an auxiliary set. Without loss of generality, substitute the first element of \eqref{3.9} with 1, producing an $n_k$-dimensional vector
\begin{equation}
\mathbf{P}_k(X_k):=(1,B_{2,p_k,\boldsymbol{\xi}_k}^k(X_k),\ldots,B_{n_k,p_k,\boldsymbol{\xi}_k}^k(X_k))^\intercal
\label{4.1}
\end{equation}
of the elements of the auxiliary set. The auxiliary B-splines in \eqref{4.1} are also linearly independent \cite{rahman20}.

\item
Assemble an $n_k \times n_k$ spline moment matrix
\begin{equation}
\mathbf{G}_k:=\mathbb{E}[\mathbf{P}_k(X_k) \mathbf{P}_k^\intercal(X_k)],
\label{4.5}
\end{equation}
which exists because $X_k$ has finite moments up to order $2 p_k$, as mandated by Assumption \ref{a1}. Furthermore, it is symmetric and positive-definite, ensuring the Cholesky factorization: $\mathbf{G}_k = \mathbf{Q}_k \mathbf{Q}_k^\intercal$, where $\mathbf{Q}_k$ is an $n_k \times n_k$ lower-triangular matrix.

\item
Employ a whitening transformation to generate an $n_k$-dimensional vector of orthonormalized B-splines
\begin{equation}
\boldsymbol{\psi}_k(X_k)= \mathbf{Q}_k \mathbf{P}_k(X_k),
\label{4.7}
\end{equation}
consisting of the components $\psi_{i_k,p_k,\boldsymbol{\xi}_k}^k(X_k)$, $i_k=1,\ldots,n_k$, $k=1,\ldots,N$.

\end{enumerate}

The whitening transformation in \eqref{4.7} is a linear transformation that alters $\mathbf{P}_k(X_k)$ into $\boldsymbol{\psi}_k(X_k)$ in such a way that the latter has uncorrelated random splines.  The transformation is called ``whitening'' because it converts one random vector to the other, which has statistical properties akin to that of a white noise vector \cite{kessy16,rahman18b}.

\begin{proposition}
Given the preamble of Proposition \ref{p2} and linear independence of auxiliary B-splines, the set of elements of $\boldsymbol{\psi}_k(x_k)$ from \eqref{4.7} also spans the spline space $\mathcal{S}_{k,p_k,\boldsymbol{\xi}_k}$, that is,
\begin{equation}
\mathcal{S}_{k,p_k,\boldsymbol{\xi}_k} :=
\operatorname{span}\{\psi_{i_k,p_k,\boldsymbol{\xi}_k}^k(x_k)\}_{i_k=1,\ldots,n_k}.
\label{4.9}
\end{equation}
\label{p5}
\end{proposition}

Figure \ref{fig1} presents a set of six second-order ($p=2$) B-spline functions on $[-1,1]$ with the uniformly spaced knot sequence $\boldsymbol{\xi}=\{-1,-1,-1,-0.5,0,0.5,1,1,1\}$ before and after orthonormalization.  The standard B-splines in Figure \ref{fig1}(a) are derived from the Cox-de Boor formula. They are non-negative and locally supported but not orthonormal with respect to the probability measure of the random variable $X$ defined on $[-1,1]$. Consider three cases of $X$ following uniform, truncated Gaussian, and Beta measures with their probability densities $f_X:[-1,1] \to \mathbb{R}\setminus(-\infty,0]$ defined by
\[
f_X(x) =
\begin{cases}
\displaystyle
\frac{1}{2},                                           & \text{(uniform)},  \\[12pt]
\displaystyle
\frac{2\phi(2x+1)}{\Phi(3)-\Phi(-1)},                  & \text{(truncated Gaussian)}, \\[12pt]
\displaystyle
\frac{\Gamma(5)(x+1)^2(1-x)}{16\Gamma(3)\Gamma(2)},    & \text{(Beta)}.
\end{cases}
\]
Here, $\Phi(\cdot)$ and $\phi(\cdot)$ are the probability distribution and density functions, respectively, of a standard Gaussian random variable. Given these probability densities and the Cholesky factorization of respective spline moment matrices, Figures \ref{fig1}(b) through \ref{fig1}(d) describe the associated orthonormalized B-splines. They depend not only on the spacing of knots but also on the probability measure of $X$.  Note that after orthonormalization, the non-constant B-splines are neither non-negative nor locally supported. Having orthonormal basis functions, however, is essential before proceeding with SDD.
\begin{figure}[htbp]
\begin{centering}
\includegraphics[width=0.73\textwidth]{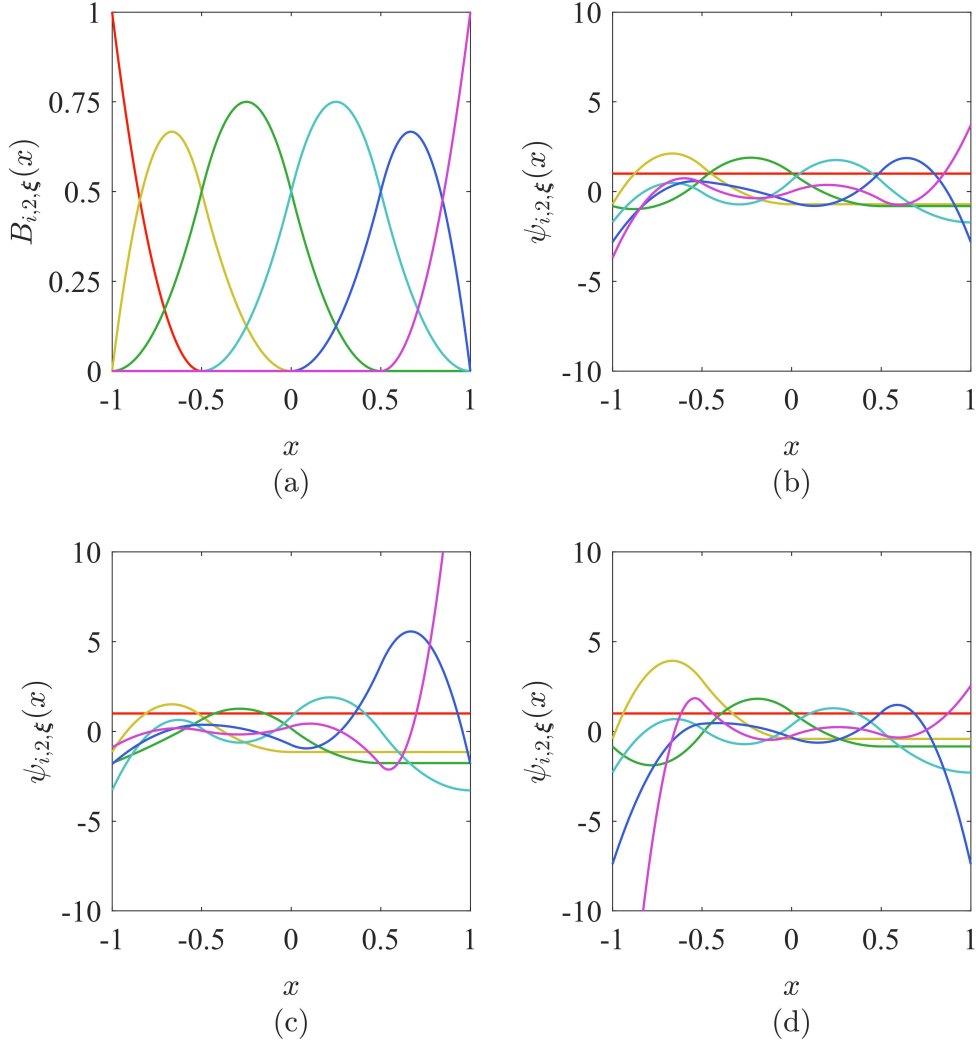}
\par\end{centering}
\caption{A set of B-splines associated with the knot sequence $\boldsymbol{\xi}=\{-1,-1,-1,-0.5,0,0.5,1,1,1\}$ and order $p=2$;
(a) non-orthonormal basis;
(b) orthonormal basis for uniform measure;
(c) orthonormal basis for truncated Gaussian measure;
(d) orthonormal basis for Beta measure.
}
\label{fig1}
\end{figure}

\subsection{Statistical properties}
Similar to $\mathbf{P}_k(X_k)$, $\boldsymbol{\psi}_k(X_k)$ is also a function of random input variable $X_k$. The first- and second-moment properties of $\boldsymbol{\psi}_k(X_k)$ are \cite{rahman20}

\begin{equation}
\mathbb{E} \left[ \boldsymbol{\psi}_k(X_k) \right] = (1, 0, \ldots, 0)^\intercal
\label{4.10}
\end{equation}
and
\begin{equation}
\mathbb{E} \left[ \boldsymbol{\psi}_k(X_k) \boldsymbol{\psi}_k^\intercal(X_k) \right] = \mathbf{I}_{n_k},
\label{4.11}
\end{equation}
respectively, where $\mathbf{I}_{n_k}$ is the $n_k \times n_k$ identity matrix.

The orthonormalized B-splines engender Fourier-like series expansion in a Hilbert space, resulting in succinct forms of the expansion and second-moment properties of an output random variable of interest.

\section{Multivariate orthonormal splines}
From Assumption \ref{a1}, the joint probability density function of the input vector $\mathbf{X}=(X_{1},\ldots,X_{N})^\intercal$ is the product of marginal density functions.  Therefore, measure-consistent multivariate orthonormal B-splines in $\mathbf{x}$ are easily constructed \cite{rahman20}, but their number will skyrocket if $N$ is too large. To circumvent the escalation, appropriate dimensionwise tensor products of measure-consistent univariate B-splines are proposed instead.

For $N\in\mathbb{N}$, denote by $\{1,\ldots,N\}$ an index set, so that $u = \{k_1,\ldots,k_{|u|}\} \subseteq \{1,\ldots,N\}$ is a subset, including the empty set $\emptyset$, with cardinality $0 \le |u| \le N$.  Let $\mathbf{X}_u:=(X_{k_1},\ldots,X_{k_{|u|}})^\intercal$, a subvector of $\mathbf{X}$, be defined on the abstract probability space $(\Omega^u,\mathcal{F}^u,\mathbb{P}^u)$, where $\Omega^u$ is the sample space of $\mathbf{X}_u$, $\mathcal{F}^u$ is a $\sigma$-algebra on $\Omega^u$, and $\mathbb{P}^u$ is a probability measure.  The corresponding image probability space is $(\mathbb{A}^u,\mathcal{B}^{u},f_{\mathbf{X}_u}d\mathbf{x}_u)$, where $\mathbb{A}^u:=\times_{k \in u} \mathbb{A}^{\{k\}} \subset \mathbb{R}^{|u|}$ is the image sample space of $\mathbf{X}_u$, $\mathcal{B}^{u}$ is the Borel $\sigma$-algebra on $\mathbb{A}^u$, and $f_{\mathbf{X}_u}(\mathbf{x}_u)$ is the marginal probability density function of $\mathbf{X}_u$ supported on $\mathbb{A}^u$.  Under Assumption \ref{a1},
\[
f_{\mathbf{X}_u}(\mathbf{x}_u)=
\displaystyle
\prod_{k \in u} f_{X_k}(x_k) =
\prod_{l=1}^{|u|} f_{X_{k_l}}(x_{k_l}),~
\mathbf{x}_u:=(x_{k_1},\ldots,x_{k_{|u|}})^\intercal.
\]

\subsection{Dimensionwise orthonormalized B-splines}
For each $k \in u \ne \emptyset$, suppose the knot sequence $\boldsymbol{\xi}_k$ on the interval $\mathbb{A}^{\{k\}}=[a_k,b_k]$, number of basis functions $n_k$, and degree $p_k$ have been specified.  The associated vector of measure-consistent univariate orthonormal splines in $x_k$ is
\[
\boldsymbol{\psi}_k(x_k):=(\psi_{1,p_k,\boldsymbol{\xi}_k}^k(x_k),\ldots,
\psi_{n_k,p_k,\boldsymbol{\xi}_k}^k(x_k))^\intercal,
~~k \in u.
\]
Correspondingly, the spline space is $\mathcal{S}_{k,p_k,\boldsymbol{\xi}_k}$, as expressed by \eqref{3.7}.  To define tensor-product B-splines in $\mathbf{x}_u$
and the associated spline space, define a multi-index $\mathbf{p}_u:=(p_{k_1},\ldots,p_{k_{|u|}}) \in \mathbb{N}_0^{|u|}$, representing the degrees of splines in all $|u|$ coordinate directions.  Denote by
$\boldsymbol{\Xi}_u:=\{ \boldsymbol{\xi}_{k_1},\ldots,\boldsymbol{\xi}_{k_{|u|}} \}$ a family of all $|u|$ knot sequences.  Because of the tensor nature of the resulting space, many properties of univariate splines carry over, described as follows.

\begin{definition}
Given $\mathbf{p}_u:=(p_{k_1},\ldots,p_{k_{|u|}}) \in \mathbb{N}_0^{|u|}$ and $\boldsymbol{\Xi}_u:=\{ \boldsymbol{\xi}_{k_1},\ldots,\boldsymbol{\xi}_{k_{|u|}} \}$, the tensor-product spline space, denoted by $\mathcal{S}_{\mathbf{p}_u,\boldsymbol{\Xi}_u}^u$, is defined by
\begin{equation}
\mathcal{S}_{\mathbf{p}_u,\boldsymbol{\Xi}_u}^u :=
\bigotimes_{k \in u} \mathcal{S}_{k,p_k,\boldsymbol{\xi}_k}=
\bigotimes_{l=1}^{|u|} \mathcal{S}_{k_l,p_{k_l},\boldsymbol{\xi}_{k_l}},
\label{5.1}
\end{equation}
\label{d8}
where the symbol $\bigotimes$ implies tensor product.
\end{definition}

It is clear from Definition \ref{d8} that $\mathcal{S}_{\mathbf{p}_u,\boldsymbol{\Xi}_u}^u$ is a linear space of dimension
\[
\displaystyle
\prod_{k\in u} n_k,
\]
where $n_k$, the dimension of the spline space $\mathcal{S}_{k,p_k,\boldsymbol{\xi}_k}$, is obtained from \eqref{3.8} when each knot sequence is chosen according to \eqref{3.2}.  Each spline function $g_u \in \mathcal{S}_{\mathbf{p}_u,\boldsymbol{\Xi}_u}^u$ is defined on a $|u|$-dimensional rectangular domain
\[
\mathbb{A}^{u}:=\displaystyle \bigtimes_{k \in u} \mathbb{A}^{k} =
\displaystyle \bigtimes_{k \in u} [a_k,b_k].
\]
Define another multi-index $\mathbf{i}_u:=(i_{k_1},\ldots,i_{k_{|u|}})$ and denote by
\[
\mathcal{I}_{u,\mathbf{n}_u} :=
\left\{ \mathbf{i}_u=(i_{k_1},\ldots,i_{k_{|u|}}): 1 \le i_{k_l} \le n_{k_l}, l=1,\ldots,|u| \right\}
\subset \mathbb{N}^{|u|}
\]
the associated multi-index set.  The index set has cardinality
\[
|\mathcal{I}_{u,\mathbf{n}_u}|=\displaystyle \prod_{k \in u} n_k,
\]
which tallies with the dimension of $\mathcal{S}_{\mathbf{p}_u,\boldsymbol{\Xi}_u}^u$.  For the coordinate direction $k_l$, define by
\begin{equation}
I_{k_l}=r_{k_l}-1,
\end{equation}
the number of subintervals corresponding to the knot sequence $\boldsymbol{\xi}_{k_l}$ with $r_{k_l}$ distinct knots. Then the partition, defined by the knot sequences $\boldsymbol{\xi}_{k_1},\ldots,\boldsymbol{\xi}_{k_{|u|}}$, splits the $|u|$-dimensional rectangle $\mathbb{A}^{u}:=\times_{k \in u}[a_k,b_k]$ into smaller rectangles
\begin{equation}
\left\{
\mathbf{x}_u=(x_{k_1},\ldots,x_{k_{|u|}})
: \zeta_{k_l,j_{k_l}} \le x_{k_l} < \zeta_{k_l,j_{k_l+1}}, ~l=1,\ldots,|u|
\right\},~j_{k_l} = 1,\ldots,I_{k_l},
\end{equation}
where $\zeta_{k_l,j_{k_l}}$ is the $j_{k_l}$th distinct knot in the coordinate direction $k_l$. A mesh is defined by the partition of $\mathbb{A}^u$ into such rectangular elements. Define the largest element size in each coordinate direction $k_l$ by
\[
h_{u,{k_l}} := \max_{j_{k_l} = 1, \ldots, I_{k_l}} \left( \zeta_{k_l,j_{k_l}+1}-\zeta_{k_l,j_{k_l}} \right),~l=1,\ldots,|u|.
\]
Then, given the knot sequences $\boldsymbol{\Xi}_u =\{ \boldsymbol{\xi}_{k_1},\ldots,\boldsymbol{\xi}_{k_{|u|}} \}$,
\[
\mathbf{h}_u:=(h_{u,{k_1}},\ldots,h_{u,{k_{|u|}}})~~\text{and}~~
h_u := \max_{l=1,\ldots,|u|} h_{u,{k_l}}
\]
define a vector of the largest element sizes in all $|u|$ coordinates and the global mesh size, respectively, for the domain $\mathbb{A}^u$.

Given the univariate B-splines in all $|u|$ coordinate directions, a formal definition of tensor-product multivariate B-splines in $\mathbf{x}_u$ is defined as follows.

\begin{definition}
Let $\mathbf{X}:=(X_{1},\ldots,X_{N})^\intercal:(\Omega,\mathcal{F})\to(\mathbb{A}^{N},\mathcal{B}^{N})$ be a vector of $N \in \mathbb{N}$ input random variables fulfilling Assumption \ref{a1}. Given a specified degree $\mathbf{p}=(p_{1},\ldots,p_{N}) \in \mathbb{N}_0^{|u|}$ and family of knot sequences $\boldsymbol{\Xi}=\{ \boldsymbol{\xi}_{1},\ldots,\boldsymbol{\xi}_{N} \}$, suppose the univariate orthonormal B-splines consistent with the marginal probability measures in all coordinate directions have been obtained as the sets
$\{\psi_{1,p_k,\boldsymbol{\xi}_k}^k(x_k),\ldots,\psi_{n_k,p_k,\boldsymbol{\xi}_k}^k(x_k)\}$,
$k=1,\ldots,N$. Then, for $\emptyset \ne u = \{k_1,\ldots,k_{|u|}\} \subseteq \{1,\ldots,N\}$, with
$\mathbf{p}_u=(p_{k_1},\ldots,p_{k_{|u|}}) \in \mathbb{N}_0^{|u|}$ and $\boldsymbol{\Xi}_u=\{ \boldsymbol{\xi}_{k_1},\ldots,\boldsymbol{\xi}_{k_{|u|}} \}$ in mind,
the multivariate orthonormal B-splines in $\mathbf{x}_u=(x_{k_1},\ldots,x_{k_{|u|}})$ consistent with the probability measure $f_{\mathbf{X}_u}(\mathbf{x}_u)d\mathbf{x}_u$ is
\begin{equation}
\Psi_{\mathbf{i}_u,\mathbf{p}_u,\boldsymbol{\Xi}_u}^u(\mathbf{x}_u) =
\displaystyle
\prod_{k \in u}
\psi_{i_k,p_k,\boldsymbol{\xi}_k}^k(x_k) =
\displaystyle
\prod_{l=1}^{|u|}
\psi_{i_{k_l},p_{k_l},\boldsymbol{\xi}_{k_l}}^k(x_{k_l}),
~~\mathbf{i}_u=(i_{k_1},\ldots,i_{k_{|u|}}) \in \mathcal{I}_{u,\mathbf{n}_u}.
\label{5.2}
\end{equation}
\label{d9}
\end{definition}

\subsection{Dimensionwise orthogonal decomposition of spline spaces}
An orthogonal decomposition of spline spaces entailing dimensionwise splitting leads to SDD.  Here, to facilitate such splitting of the spline space $\mathcal{S}_{\mathbf{p}_u,\boldsymbol{\Xi}_u}^u$ for any $\emptyset \ne u = \{k_1,\ldots,k_{|u|}\} \subseteq \{1,\ldots,N\}$, limit the index $i_{k_l}$, $l=1,\ldots,|u|$, associated with the $k_l$-th variable $x_{k_l}$, to run from 2 to $n_{k_l}$.  That is, remove the first constant element of $\boldsymbol{\psi}_k(x_k)$ to form a reduced basis
\[
\left\{\psi_{2,p_k,\boldsymbol{\xi}_k}^k(x_k),\ldots,\psi_{n_k,p_k,\boldsymbol{\xi}_k}^k(x_k)\right \},
~~k \in u,
\]
spanning the corresponding spline space
\[
\bar{\mathcal{S}}_{k,p_k,\boldsymbol{\xi}_k} :=
\operatorname{span} \left\{ \psi_{i_k,p_k,\boldsymbol{\xi}_k}^k(x_k) \right\}_{i_k=2,\ldots,n_k}
\subset \mathcal{S}_{k,p_k,\boldsymbol{\xi}_k}.
\]
Here, $\bar{\mathcal{S}}_{k,p_k,\boldsymbol{\xi}_k}$ is a subspace of $\mathcal{S}_{k,p_k,\boldsymbol{\xi}_k}$ of functions, which have \emph{zero} means because each element of the reduced basis has a \emph{zero} mean, as per \eqref{4.10}.  It is orthogonal to the subspace $\boldsymbol{1}:=\operatorname{span}\{1\}$ of constant functions.  This induces an orthogonal decomposition of
\begin{equation}
\mathcal{S}_{k,p_k,\boldsymbol{\xi}_k} = \boldsymbol{1} \oplus \bar{\mathcal{S}}_{k,p_k,\boldsymbol{\xi}_k},
\label{5.3}
\end{equation}
good for any $k \in u$.  From \eqref{4.10} and \eqref{4.11}, any two distinct subspaces $\bar{\mathcal{S}}_{k,p_k,\boldsymbol{\xi}_k}$ and $\bar{\mathcal{S}}_{l,p_l,\boldsymbol{\xi}_l}$ are orthogonal to each other whenever $k \ne l$ or $i_k \ne i_l$.  Therefore, combining \eqref{5.1} and \eqref{5.3}, the spline space $\mathcal{S}_{\mathbf{p}_u,\boldsymbol{\Xi}_u}^u$ becomes
\begin{equation}
\mathcal{S}_{\mathbf{p}_u,\boldsymbol{\Xi}_u}^u =
\displaystyle
\bigotimes_{k \in u} \left( \boldsymbol{1} \oplus \bar{\mathcal{S}}_{k,p_k,\boldsymbol{\xi}_k} \right) =
\boldsymbol{1} \oplus
\displaystyle
\bigoplus_{\emptyset \ne v \subseteq u}~
\bigotimes_{k \in v} \bar{\mathcal{S}}_{k,p_k,\boldsymbol{\xi}_k} =
\boldsymbol{1} \oplus
\displaystyle
\bigoplus_{\emptyset \ne v \subseteq u}
\bar{\mathcal{S}}_{\mathbf{p}_v,\boldsymbol{\Xi}_v}^v,
\label{5.4}
\end{equation}
with the symbol $\oplus$ representing an orthogonal sum and
\begin{equation}
\bar{\mathcal{S}}_{\mathbf{p}_v,\boldsymbol{\Xi}_v}^v :=
\bigotimes_{k \in v} \bar{\mathcal{S}}_{k,p_k,\boldsymbol{\xi}_k}
\label{5.5}
\end{equation}
defining a subspace of $\mathcal{S}_{\mathbf{p}_v,\boldsymbol{\Xi}_v}^v$.

Since the decomposition in \eqref{5.4} is valid for any $\emptyset \ne u \subseteq \{1,\ldots,N\}$, set $u = \{1,\ldots,N\}$ in \eqref{5.4} to obtain $\mathcal{S}_{\mathbf{p},\boldsymbol{\Xi}}$, which defines the space of all real-valued splines associated with $\mathbf{p}=(p_1,\ldots,p_N) \in \mathbb{N}_0^N$ and
$\boldsymbol{\Xi}=\{ \boldsymbol{\xi}_1,\ldots,\boldsymbol{\xi}_N \}$.  Then, swapping $v$ for $u$ in \eqref{5.4}
produces yet another orthogonal decomposition of
\begin{equation}
\mathcal{S}_{\mathbf{p},\boldsymbol{\Xi}}=
\displaystyle
\bigotimes_{k=1}^N \left( \boldsymbol{1} \oplus \bar{\mathcal{S}}_{k,p_k,\boldsymbol{\xi}_k} \right) =
\boldsymbol{1} \oplus
\displaystyle
\bigoplus_{\emptyset \ne u \subseteq \{1,\ldots,N\}}~
\bigotimes_{k=1}^N \bar{\mathcal{S}}_{k,p_k,\boldsymbol{\xi}_k} =
\boldsymbol{1} \oplus
\displaystyle
\bigoplus_{\emptyset \ne u \subseteq \{1,\ldots,N\}}
\bar{\mathcal{S}}_{\mathbf{p}_u,\boldsymbol{\Xi}_u}^u,
\label{5.6}
\end{equation}
comprising $2^N$ distinct subspaces.  Accordingly, a spline function $g \in \mathcal{S}_{\mathbf{p},\boldsymbol{\Xi}}$ is also decomposed as
\begin{equation}
g(\mathbf{x}) =  g_{\emptyset} +
\displaystyle \sum_{\emptyset \ne u\subseteq\{1,\ldots,N\}}g_{u}(\mathbf{x}_{u})
\label{5.7}
\end{equation}
into $2^N$ distinct terms, where $g_{\emptyset} \in \boldsymbol{1}$ and $g_{u}(\mathbf{x}_{u}) \in \bar{\mathcal{S}}_{\mathbf{p}_u,\boldsymbol{\Xi}_u}^u$ are various component functions, describing a constant or an $|u|$-variate interaction effect of $\mathbf{x}_{u}=(x_{k_{1}},\ldots,x_{k_{|u|}})$ on $g$ when $|u|=0$ or $|u|>0$.  The expansion in \eqref{5.7}, referred to as dimensional decomposition by Rahman \cite{rahman14}, is unique as long as the probability measure of $\mathbf{X}$ is fixed.  The decomposition, first presented by Hoeffding \cite{hoeffding48} in relation to his seminal work on $U$-statistics, has been studied by many researchers \cite{efron81,griebel10,kuo10,owen97,rabitz99,rahman14,sobol03}.

For a product-type probability measure $f_{\mathbf{X}}(\mathbf{x})d\mathbf{x}$, as assumed here, the component functions of $g$ can be obtained as \cite{rahman14}
\begin{subequations}
\begin{alignat}{2}
g_{\emptyset} = & \int_{\mathbb{A}^{N}}g(\mathbf{x})f_{\mathbf{X}}(\mathbf{x})d\mathbf{x},
\label{5.8a} \\
g_{u}(\mathbf{x}_{u}) = &
\displaystyle \int_{\mathbb{A}^{N-|u|}}g(\mathbf{x}_{u},\mathbf{x}_{-u})
f_{\mathbf{X}_{-u}}(\mathbf{x}_{-u})d\mathbf{x}_{-u}-
\displaystyle \sum_{v\subset u}g_{v}(\mathbf{X}_{v}),
\label{5.8b}
\end{alignat}
\end{subequations}
where $-u=\{1,\ldots,N\}\setminus u$ is a complementary set of $u\subseteq\{1,\ldots,N\}$ and $f_{\mathbf{X}_{-u}}(\mathbf{x}_{-u})$ is the marginal density function of $\mathbf{X}_{-u}$.  The decomposition in \eqref{5.7} with its component functions obtained from \eqref{5.8a} and \eqref{5.8b} is the well-known analysis-of-variance (ANOVA) decomposition \cite{efron81,sobol03}.  Readers interested in other variants of dimensional decomposition, such as those obtained for a Dirac measure and a non-product-type probability measure, are referred to the works of Griebel and Holtz \cite{griebel10} and Rahman \cite{rahman14b}, respectively.

\subsection{Statistical properties}
When the input random variables $X_1,\ldots,X_N$, instead of real variables $x_1,\ldots,x_N$, are inserted in the argument, the multivariate splines $\Psi_{\mathbf{i}_u,\mathbf{p}_u,\boldsymbol{\Xi}_u}^u(\mathbf{X}_u)$, $\emptyset \ne u\subseteq\{1,\ldots,N\}$, $\mathbf{i}_u \in \mathcal{I}_{u,\mathbf{n}_u}$, become functions of random input variables.  Therefore, it is important to establish their second-moment properties, to be exploited in Section 6.  However, as the dimensionwise decomposition, described in the previous subsection, excludes the first constant element of $\boldsymbol{\psi}_{k_l}(x_{k_l})$, the definition of a reduced index set, expressed by
\[
\bar{\mathcal{I}}_{u,\mathbf{n}_u} := \left\{ \mathbf{i}_u=(i_{k_1},\ldots,i_{k_{|u|}}): 2 \le i_{k_l} \le n_{k_l}, 1 \le l \le |u| \right\}
\subset (\mathbb{N}\setminus \{1\})^{|u|},
\]
is required. The principal difference between $\bar{\mathcal{I}}_{u,\mathbf{n}_u}$ and $\mathcal{I}_{u,\mathbf{n}_u}$ stems from the fact that the former index set restricts the range of index $i_{k_l}$, $l=1,\ldots,|u|$, to vary from 2 to $n_{k_l}$. As a result, any reduction of the degree of interaction of the corresponding multivariate spline basis below $|u|$ is avoided.

The reduced index set has cardinality
\[
|\bar{\mathcal{I}}_{u,\mathbf{n}_u}| := \prod_{k \in u} (n_k-1).
\]
It is elementary to verify that the sum
\[
\displaystyle
1 + \sum_{\emptyset \ne u \subseteq \{1,\ldots,N\}} |\bar{\mathcal{I}}_{u,\mathbf{n}_u}| =
1 + \sum_{\emptyset \ne u \subseteq \{1,\ldots,N\}} \prod_{k \in u} (n_k-1) =
\prod_{k=1}^N n_k,
\]
thus matching the dimension of $\mathcal{S}_{\mathbf{p},\boldsymbol{\Xi}}$.

\begin{proposition}
Let $\mathbf{X}:=(X_{1},\ldots,X_{N})^\intercal:(\Omega,\mathcal{F})\to(\mathbb{A}^{N},\mathcal{B}^{N})$ be a vector of $N \in \mathbb{N}$ input random variables fulfilling Assumption \ref{a1}.  Then, for $\emptyset \ne u,v \subseteq \{1,\ldots,N\}$, $\mathbf{i}_u \in \bar{\mathcal{I}}_{u,\mathbf{n}_u}$, and $\mathbf{j}_v \in \bar{\mathcal{I}}_{v,\mathbf{n}_v}$, the first- and second-order moments of multivariate orthonormal B-splines are
\begin{equation}
\mathbb{E} \left[ \Psi_{\mathbf{i}_u,\mathbf{p}_u,\boldsymbol{\Xi}_u}^u(\mathbf{X}_u) \right] = 0
\label{5.9}
\end{equation}
and
\begin{equation}
\mathbb{E} \left[ \Psi_{\mathbf{i}_u,\mathbf{p}_u,\boldsymbol{\Xi}_u}^u(\mathbf{X}_u)
\Psi_{\mathbf{j}_v,\mathbf{p}_v,\boldsymbol{\Xi}_v}^v(\mathbf{X}_v) \right] =
\begin{cases}
1,              & u = v~\text{and}~\mathbf{i}_u=\mathbf{j}_v, \\
0,              & \text{otherwise},
\end{cases}
\label{5.10}
\end{equation}
respectively.
\label{p7}
\end{proposition}

The statistical properties of univariate orthonormal B-splines in \eqref{4.10} and \eqref{4.11} lead to the desired result of Proposition \ref{p7}.

\subsection{Orthonormal basis}
The following proposition shows that a reduced set of multivariate orthonormal splines from Definition \ref{d9} span the spline space of interest.

\begin{proposition}
Let $\mathbf{X}:=(X_{1},\ldots,X_{N})^\intercal:(\Omega,\mathcal{F})\to(\mathbb{A}^{N},\mathcal{B}^{N})$ be a vector of $N \in \mathbb{N}$ input random variables fulfilling Assumption \ref{a1} and $\mathbf{X}_u:=(X_{k_1},\ldots,X_{k_{|u|}})^\intercal:(\Omega^u,\mathcal{F}^u)\to(\mathbb{A}^u,\mathcal{B}^u)$, $\emptyset \ne u \subseteq \{1,\ldots,N\}$, be a subvector of $\mathbf{X}$.  Then $\{\Psi_{\mathbf{i}_u,\mathbf{p}_u,\boldsymbol{\Xi}_u}(\mathbf{x}_u): \mathbf{i}_u \in \bar{\mathcal{I}}_{u,\mathbf{n}_u}\}$, the reduced set of multivariate orthonormal B-splines for a chosen degree $\mathbf{p}_u$ and family of knot sequences $\boldsymbol{\Xi}_u$, consistent with the probability measure $f_{\mathbf{X}_u}(\mathbf{x}_u)d\mathbf{x}_u$, is a basis of $\bar{\mathcal{S}}_{\mathbf{p}_u,\boldsymbol{\Xi}_u}^u$, that is,
\begin{equation}
\bar{\mathcal{S}}_{\mathbf{p}_u,\boldsymbol{\Xi}_u}^u =
\operatorname{span} \left\{ \Psi_{\mathbf{i}_u,\mathbf{p}_u,\boldsymbol{\Xi}_u}^u(\mathbf{x}_u) \right\}_{\mathbf{i}_u \in \bar{\mathcal{I}}_{u,\mathbf{n}_u}} =
\operatorname{span} \left\{ \psi_{i_k,p_k,\boldsymbol{\xi}_k}^k(x_k) \right\}_{i_k=2,\ldots,n_k}.
\label{5.10b}
\end{equation}
Moreover,
\begin{equation}
\mathcal{S}_{\mathbf{p},\boldsymbol{\Xi}}=
\boldsymbol{1} \oplus
\displaystyle
\bigoplus_{\emptyset \ne u \subseteq \{1,\ldots,N\}}
\operatorname{span} \left\{ \Psi_{\mathbf{i}_u,\mathbf{p}_u,\boldsymbol{\Xi}_u}^u(\mathbf{x}_u) \right\}_{\mathbf{i}_u \in \bar{\mathcal{I}}_{u,\mathbf{n}_u}} =
\boldsymbol{1} \oplus
\displaystyle
\bigoplus_{\emptyset \ne u \subseteq \{1,\ldots,N\}}
\bigotimes_{k \in u}
\operatorname{span} \left\{ \psi_{i_k,p_k,\boldsymbol{\xi}}^k(x_k) \right\}_{i_k=2,\ldots,n_k}.
\label{5.11}
\end{equation}
\label{p8}
\end{proposition}

The statistical properties in Proposition \ref{p7} result in linear independence of the elements of $\{ \Psi_{\mathbf{i}_u,\mathbf{p}_u,\boldsymbol{\Xi}_u}^u(\mathbf{x}_u)\}_{\mathbf{i}_u \in \bar{\mathcal{I}}_{u,\mathbf{n}_u}}$.  The desired results can be obtained readily.

The multivariate and univariate orthonormalized B-splines described in this paper are both consistent with an arbitrary probability measure of input random variables with bounded domains. For unbounded domains, such as those associated with Gaussian random variables, appropriate probability preserving transformations are required. In this case, a random variable with an unbounded domain needs to be mapped to another random variable with a bounded domain. Readers interested in additional details and examples are referred to the authors' other works \cite{dixler21,jahanbin20}.

\section{Spline dimensional decomposition}
Given an input random vector $\mathbf{X}:=(X_{1},\ldots,X_{N})^\intercal:(\Omega,\mathcal{F})\to(\mathbb{A}^{N},\mathcal{B}^{N})$
with known probability density function $f_{\mathbf{X}}({\mathbf{x}})$ on $\mathbb{A}^N \subset \mathbb{R}^N$, designate by $y(\mathbf{X}):=y(X_{1},\ldots,X_{N})$ a real-valued, square-integrable, measurable transformation on $(\Omega, \mathcal{F})$.  Here, $y: \mathbb{A}^N \to \mathbb{R}$ represents an output function from a mathematical model, describing relevant stochastic performance of a complex system.  A principal objective of UQ is to estimate accurately, if it is not possible to determine exactly, the probabilistic characteristics of $y(\mathbf{X})$.  More often than not, $y$ is assumed to belong to a reasonably large function class, such as the well-known $L^2$ space.

Corresponding to the image probability space $(\mathbb{A}^N,\mathcal{B}^{N},f_{\mathbf{X}}d\mathbf{x})$,
let
\[
L^2(\mathbb{A}^N,\mathcal{B}^{N},f_{\mathbf{X}}d\mathbf{x}):=
\left\{
y: \mathbb{A}^N \to \mathbb{R}:
~\int_{\mathbb{A}^N} \left| y(\mathbf{x})\right|^2  f_{\mathbf{X}}({\mathbf{x}})d\mathbf{x} < \infty
\right\}
\]
define a weighted $L^2$-space of interest.  Clearly, $L^2(\mathbb{A}^N,\mathcal{B}^{N},f_{\mathbf{X}}d\mathbf{x})$ is a Hilbert space, which is endowed with an inner product
\[
\left( y(\mathbf{x}),z(\mathbf{x}) \right)_{L^2(\mathbb{A}^N,\mathcal{B}^{N},f_{\mathbf{X}}d\mathbf{x})}:=
\int_{\mathbb{A}^N} y(\mathbf{x})z(\mathbf{x})f_{\mathbf{X}}(\mathbf{x})d\mathbf{x}
\]
and induced norm
\[
\|y(\mathbf{x})\|_{L^2(\mathbb{A}^N,\mathcal{B}^{N},f_{\mathbf{X}}d\mathbf{x})}=
\sqrt{(y(\mathbf{x}),y(\mathbf{x}))_{L^2(\mathbb{A}^N,\mathcal{B}^{N},f_{\mathbf{X}}d\mathbf{x})}}.
\]
Similarly, for the abstract probability space $(\Omega,\mathcal{F},\mathbb{P})$, there is an isomorphic Hilbert space
\[
L^2(\Omega,\mathcal{F},\mathbb{P}):=
\left\{Y:\Omega \to \mathbb{R}:
~\int_{\Omega} \left|y(\mathbf{X}(\omega))\right|^2 d\mathbb{P}(\omega) < \infty \right\}
\]
of equivalent classes of output random variables $Y=y(\mathbf{X})$ with the corresponding inner product
\[
\left( y(\mathbf{X}),z(\mathbf{X}) \right)_{L^2(\Omega, \mathcal{F}, \mathbb{P})}:=
\int_{\Omega} y(\mathbf{X}(\omega))z(\mathbf{X}(\omega))d\mathbb{P}(\omega)
\]
and norm
\[
\|y(\mathbf{X})\|_{L^2(\Omega, \mathcal{F}, \mathbb{P})}:=
\sqrt{(y(\mathbf{X}),y(\mathbf{X}))_{L^2(\Omega, \mathcal{F}, \mathbb{P})}}.
\]
It is elementary to show that $y(\mathbf{X}(\omega)) \in L^2(\Omega,\mathcal{F},\mathbb{P})$ if and only if $y(\mathbf{x}) \in L^2(\mathbb{A}^N,\mathcal{B}^{N},f_{\mathbf{X}}d\mathbf{x})$.

\subsection{SDD approximation}
An SDD approximation of a square-integrable random variable $y(\mathbf{X}) \in L^2(\Omega,\mathcal{F},\mathbb{P})$ is simply its dimensionwise orthogonal projection onto the spline space $\mathcal{S}_{\mathbf{p},\boldsymbol{\Xi}}$, formally presented as follows.

\begin{theorem}
Let $\mathbf{X}:=(X_{1},\ldots,X_{N})^\intercal:(\Omega,\mathcal{F})\to(\mathbb{A}^{N},\mathcal{B}^{N})$ be a vector of $N \in \mathbb{N}$ input random variables fulfilling Assumption \ref{a1}.  Suppose the degree and family of knot sequences in all coordinate directions have been specified as $\mathbf{p}=(p_{1},\ldots,p_{N}) \in \mathbb{N}_0^{|u|}$ and $\boldsymbol{\Xi}=\{ \boldsymbol{\xi}_{1},\ldots,\boldsymbol{\xi}_{N} \}$, respectively.
For $\emptyset \ne u \subseteq \{1,\ldots,N\}$ and $\mathbf{X}_u:=(X_{k_1},\ldots,X_{k_{|u|}})^\intercal: (\Omega^u,\mathcal{F}^u)\to(\mathbb{A}^u,\mathcal{B}^u)$, with $\mathbf{p}_u=(p_{k_1},\ldots,p_{k_{|u|}}) \in \mathbb{N}_0^{|u|}$ and $\boldsymbol{\Xi}_u=\{ \boldsymbol{\xi}_{k_1},\ldots,\boldsymbol{\xi}_{k_{|u|}} \}$ in mind, denote by $\{\Psi_{\mathbf{i}_u,\mathbf{p}_u,\boldsymbol{\Xi}_u}^u(\mathbf{X}_u): \mathbf{i}_u \in \bar{\mathcal{I}}_{u,\mathbf{n}_u}\}$ a reduced set comprising multivariate orthonormal B-splines that is consistent with the probability measure $f_{\mathbf{X}_u}(\mathbf{x}_u)d\mathbf{x}_u$.  Then, for any random variable $y(\mathbf{X}) \in L^2(\Omega, \mathcal{F}, \mathbb{P})$, a hierarchically expanded Fourier-like series in multivariate orthonormal splines in $\mathbf{X}_u$, referred to as the SDD approximation of $y(\mathbf{X})$, is derived as
\begin{equation}
y_{\mathbf{p},\boldsymbol{\Xi}}(\mathbf{X}) :=
y_\emptyset + \displaystyle \sum_{\emptyset \ne u\subseteq\{1,\ldots,N\}}
\sum_{\mathbf{i}_u \in  \bar{\mathcal{I}}_{u,\mathbf{n}_u}}
C_{\mathbf{i}_u,\mathbf{p}_u,\boldsymbol{\Xi}_u}^u \Psi_{\mathbf{i}_u,\mathbf{p}_u,\boldsymbol{\Xi}_u}^u(\mathbf{X}_u),
\label{6.1}
\end{equation}
where the SDD expansion coefficients $y_\emptyset \in \mathbb{R}$ and $C_{\mathbf{i}_u,\mathbf{p}_u,\boldsymbol{\Xi}_u}^u \in \mathbb{R}$, $\emptyset \ne u\subseteq\{1,\ldots,N\}$, $\mathbf{i}_u \in \bar{\mathcal{I}}_{u,\mathbf{n}_u}$, are defined as
\begin{equation}
y_\emptyset :=
\mathbb{E} \left[  y(\mathbf{X}) \right] :=
\int_{\mathbb{A}^N} y(\mathbf{x}) f_{\mathbf{X}}(\mathbf{x})d\mathbf{x},
\label{6.2}
\end{equation}
\begin{equation}
C_{\mathbf{i}_u,\mathbf{p}_u,\boldsymbol{\Xi}_u}^u :=
\mathbb{E} \left[  y(\mathbf{X}) \Psi_{\mathbf{i}_u,\mathbf{p}_u,\boldsymbol{\Xi}_u}^u(\mathbf{X}_u) \right] :=
\int_{\mathbb{A}^N} y(\mathbf{x})
\Psi_{\mathbf{i}_u,\mathbf{p}_u,\boldsymbol{\Xi}_u}^u(\mathbf{x}_u) f_{\mathbf{X}}(\mathbf{x})d\mathbf{x},~~
\mathbf{i}_u \in \bar{\mathcal{I}}_{u,\mathbf{n}_u}.
\label{6.3}
\end{equation}
\label{t1}
\end{theorem}

\begin{proof}
If $y(\mathbf{x}) \in L^2(\mathbb{A}^N,\mathcal{B}^{N},f_{\mathbf{X}}d\mathbf{x})$, then an orthogonal projection operator
\[
P_{\mathcal{S}_{\mathbf{p},\boldsymbol{\Xi}}}: L^2(\mathbb{A}^N,\mathcal{B}^{N},f_{\mathbf{X}}d\mathbf{x})\to \mathcal{S}_{\mathbf{p},\boldsymbol{\Xi}},
\]
defined by
\begin{equation}
P_{\mathcal{S}_{\mathbf{p},\boldsymbol{\Xi}}} y :=
y_\emptyset + \displaystyle \sum_{\emptyset \ne u\subseteq\{1,\ldots,N\}}
\sum_{\mathbf{i}_u \in  \bar{\mathcal{I}}_{u,\mathbf{n}_u}}
C_{\mathbf{i}_u,\mathbf{p}_u,\boldsymbol{\Xi}_u}^u \Psi_{\mathbf{i}_u,\mathbf{p}_u,\boldsymbol{\Xi}_u}^u(\mathbf{x}_u),
\label{6.6}
\end{equation}
can be established. By definition of the random vector $\mathbf{X}$, the set
\[
\left\{
1,\Psi_{\mathbf{i}_u,\mathbf{p}_u,\boldsymbol{\Xi}_u}^u(\mathbf{X}_u):
\emptyset \ne u\subseteq\{1,\ldots,N\},~\mathbf{i}_u \in  \bar{\mathcal{I}}_{u,\mathbf{n}_u}
\right\}
\]
is a basis of the spline subspace $\mathcal{S}_{\mathbf{p},\boldsymbol{\Xi}}$ of $L^2(\Omega,\mathcal{F},\mathbb{P})$, inheriting the properties of the basis
\[
\left\{
1,\Psi_{\mathbf{i}_u,\mathbf{p}_u,\boldsymbol{\Xi}_u}^u(\mathbf{x}_u):
\emptyset \ne u\subseteq\{1,\ldots,N\},~\mathbf{i}_u \in  \bar{\mathcal{I}}_{u,\mathbf{n}_u}
\right\}
\]
of the spline subspace $\mathcal{S}_{\mathbf{p},\boldsymbol{\Xi}}$ of $L^2(\mathbb{A}^N,\mathcal{B}^{N},f_{\mathbf{X}}d\mathbf{x})$.  Here, with a certain abuse of notation, $\mathcal{S}_{\mathbf{p},\boldsymbol{\Xi}}$ is used as a set of spline functions of both real variables ($\mathbf{x}$) and random variables ($\mathbf{X}$). Therefore, \eqref{6.6} yields the expansion in \eqref{6.1}.

To derive the expressions of the expansion coefficients, define a second moment error,
\begin{equation}
e_{\text{SDD}}:= \mathbb{E}\Biggl[ y(\mathbf{X}) -
y_\emptyset -
\displaystyle \sum_{\emptyset \ne u\subseteq\{1,\ldots,N\}} \sum_{\mathbf{i}_u \in  \bar{\mathcal{I}}_{u,\mathbf{n}_u}}
C_{\mathbf{i}_u,\mathbf{p}_u,\boldsymbol{\Xi}_u}^u \Psi_{\mathbf{i}_u,\mathbf{p}_u,\boldsymbol{\Xi}_u}^u(\mathbf{X}_u) \Biggr]^2,
\label{6.7}
\end{equation}
committed by the SDD approximation.  Differentiate both sides of \eqref{6.7} with respect to $y_\emptyset$ and $C_{\mathbf{i}_u,\mathbf{p}_u,\boldsymbol{\Xi}_u}^u$, $\emptyset \ne u\subseteq\{1,\ldots,N\}$, $\mathbf{i}_u \in \bar{\mathcal{I}}_{u,\mathbf{n}_u}$, to write
\begin{equation}
\begin{array}{rcl}
\displaystyle \frac{\partial e_{\text{SDD}}}{\partial y_\emptyset}
& = &
\displaystyle
\frac{\partial}{\partial y_\emptyset}
\mathbb{E} \Biggl[  y(\mathbf{X}) - y_\emptyset - \displaystyle \sum_{\emptyset \ne v\subseteq\{1,\ldots,N\}} \sum_{\mathbf{i}_v \in  \bar{\mathcal{I}}_{v,\mathbf{n}_v}}
C_{\mathbf{i}_v,\mathbf{p}_v,\boldsymbol{\Xi}_v}^v \Psi_{\mathbf{i}_v,\mathbf{p}_v,\boldsymbol{\Xi}_v}^v(\mathbf{X}_v) \Biggr]^2
 \\
&  = &
\displaystyle
\mathbb{E} \Biggl[
\frac{\partial}{\partial y_\emptyset} \Biggl\{
y(\mathbf{X}) -  y_\emptyset - \displaystyle \sum_{\emptyset \ne v\subseteq\{1,\ldots,N\}} \sum_{\mathbf{i}_v \in  \bar{\mathcal{I}}_{v,\mathbf{n}_v}}
C_{\mathbf{i}_v,\mathbf{p}_v,\boldsymbol{\Xi}_v}^v \Psi_{\mathbf{i}_v,\mathbf{p}_v,\boldsymbol{\Xi}_v}^v(\mathbf{X}_v)
  \Biggr\}^2
\Biggr]
 \\
&  = &
\displaystyle
 2 \mathbb{E} \Biggl[
 \Biggl\{
 y_\emptyset + \displaystyle \sum_{\emptyset \ne v\subseteq\{1,\ldots,N\}} \sum_{\mathbf{i}_v \in  \bar{\mathcal{I}}_{v,\mathbf{n}_v}}
C_{\mathbf{i}_v,\mathbf{p}_v,\boldsymbol{\Xi}_v}^v \Psi_{\mathbf{i}_v,\mathbf{p}_v,\boldsymbol{\Xi}_v}^v(\mathbf{X}_v) - y(\mathbf{X}) \Biggr\}
\times 1
 \Biggr]
 \\
&  = &
\displaystyle
2 \left\{ y_\emptyset - \mathbb{E}\left[ y(\mathbf{X})\right] \right\}
\end{array}
\label{6.8}
\end{equation}
and
\begin{equation}
\begin{array}{rcl}
\displaystyle \frac{\partial e_{\text{SDD}}}{\partial C_{\mathbf{i}_u,\mathbf{p}_u,\boldsymbol{\Xi}_u}^u}
& = &
\displaystyle
\frac{\partial}{\partial C_{\mathbf{i}_u,\mathbf{p}_u,\boldsymbol{\Xi}_u}^u}
\mathbb{E} \Biggl[  y(\mathbf{X}) - y_\emptyset - \displaystyle \sum_{\emptyset \ne v\subseteq\{1,\ldots,N\}} \sum_{\mathbf{i}_v \in  \bar{\mathcal{I}}_{v,\mathbf{n}_v}}
C_{\mathbf{i}_v,\mathbf{p}_v,\boldsymbol{\Xi}_v}^v \Psi_{\mathbf{i}_v,\mathbf{p}_v,\boldsymbol{\Xi}_v}^v(\mathbf{X}_v) \Biggl]^2
 \\
&  = &
\displaystyle
\mathbb{E} \Biggl[
\frac{\partial}{\partial C_{\mathbf{i}_u,\mathbf{p}_u,\boldsymbol{\Xi}_u}^u} \Biggl\{
y(\mathbf{X}) -  y_\emptyset - \displaystyle \sum_{\emptyset \ne v\subseteq\{1,\ldots,N\}} \sum_{\mathbf{i}_v \in  \bar{\mathcal{I}}_{v,\mathbf{n}_v}}
C_{\mathbf{i}_v,\mathbf{p}_v,\boldsymbol{\Xi}_v}^v \Psi_{\mathbf{i}_v,\mathbf{p}_v,\boldsymbol{\Xi}_v}^v(\mathbf{X}_v)
  \Biggr\}^2
\Biggr]
 \\
&  = &
\displaystyle
 2 \mathbb{E} \Biggl[
 \Biggl\{
 y_\emptyset + \displaystyle \sum_{\emptyset \ne v\subseteq\{1,\ldots,N\}} \sum_{\mathbf{i}_v \in  \bar{\mathcal{I}}_{v,\mathbf{n}_v}}
C_{\mathbf{i}_v,\mathbf{p}_v,\boldsymbol{\Xi}_v}^v \Psi_{\mathbf{i}_v,\mathbf{p}_v,\boldsymbol{\Xi}_v}^v(\mathbf{X}_v) - y(\mathbf{X}) \Biggr\}
\Psi_{\mathbf{i}_u,\mathbf{p}_u,\boldsymbol{\Xi}_u}^u(\mathbf{X}_u)
 \Biggr]
 \\
&  = &
\displaystyle
2\left\{ C_{\mathbf{i}_u,\mathbf{p}_u,\boldsymbol{\Xi}_u}^u - \mathbb{E}\left[ y(\mathbf{X})
\Psi_{\mathbf{i}_u,\mathbf{p}_u,\boldsymbol{\Xi}_u}^u(\mathbf{X}_u) \right] \right\}.
\end{array}
\label{6.9}
\end{equation}
Here, the second, third, and last lines of both \eqref{6.8} and \eqref{6.9} are obtained by interchanging the differential and expectation operators, performing the differentiation, and swapping the expectation and summation operators and applying \eqref{4.10} and \eqref{4.11}, respectively.   Setting ${\partial e_{\text{SDD}}}/{\partial y_\emptyset}=0$ in \eqref{6.8} and ${\partial e_{\text{SDD}}}/{\partial C_{\mathbf{i}_u,\mathbf{p}_u,\boldsymbol{\Xi}_u}^u}=0$ in \eqref{6.9} yields \eqref{6.2} and \eqref{6.3}, respectively.
\end{proof}

\begin{proposition}
The SDD approximation $y_{\mathbf{p},\boldsymbol{\Xi}}(\mathbf{X})$ is the best approximation of $y(\mathbf{X})$ in the sense that
\begin{equation}
\displaystyle
\mathbb{E}\left[ y(\mathbf{X})-y_{\mathbf{p},\boldsymbol{\Xi}}(\mathbf{X}) \right]^2 =
\inf_{g \in \mathcal{S}_{\mathbf{p},\boldsymbol{\Xi}}}
\mathbb{E}\left[ y(\mathbf{X})-g(\mathbf{X}) \right]^2
\label{6.4}
\end{equation}
or, equivalently,
\begin{equation}
\displaystyle
{\left\| y(\mathbf{x})-y_{\mathbf{p},\boldsymbol{\Xi}}(\mathbf{x}) \right\|}_{L^2(\mathbb{A}^N,\mathcal{B}^{N},f_{\mathbf{X}}d\mathbf{x})} =
\inf_{g \in \mathcal{S}_{\mathbf{p},\boldsymbol{\Xi}}}
{\left\| y(\mathbf{x})-g(\mathbf{x}) \right\|}_{L^2(\mathbb{A}^N,\mathcal{B}^{N},f_{\mathbf{X}}d\mathbf{x})}.
\label{6.5}
\end{equation}
\label{p9}
\end{proposition}

\begin{proof}
Any spline function $g \in \mathcal{S}_{\mathbf{p},\boldsymbol{\Xi}}$ can be expressed by
\begin{equation}
g(\mathbf{X}) =
\bar{y}_\emptyset + \displaystyle \sum_{\emptyset \ne u\subseteq\{1,\ldots,N\}}
\sum_{\mathbf{i}_u \in  \bar{\mathcal{I}}_u}
\bar{C}_{\mathbf{i}_u,\mathbf{p}_u,\boldsymbol{\Xi}_u}^u \Psi_{\mathbf{i}_u,\mathbf{p}_u,\boldsymbol{\Xi}_u}^u(\mathbf{X}_u)
\label{6.10}
\end{equation}
with some real-valued coefficients $\bar{y}_\emptyset$ and $\bar{C}_{\mathbf{i}_u,\mathbf{p}_u,\boldsymbol{\Xi}_u}^u$, $\emptyset \ne u\subseteq\{1,\ldots,N\}$, $\mathbf{i}_u \in \bar{\mathcal{I}}_u$.  To minimize $\mathbb{E}[ \{y(\mathbf{X})-g(\mathbf{X})\}^2]$, its derivatives with respect to the coefficients must be \emph{zero}, that is,
\begin{equation}
\displaystyle
\frac{\partial}{\partial \bar{y}_\emptyset}
\mathbb{E} \left[ \{y(\mathbf{X}) - g(\mathbf{X}) \}^2 \right] =
\displaystyle
\frac{\partial}{\partial \bar{C}_{\mathbf{i}_u,\mathbf{p}_u,\boldsymbol{\Xi}_u}^u}
\mathbb{E} \left[ \{y(\mathbf{X}) - g(\mathbf{X}) \}^2 \right] =
 0,
~~\mathbf{i}_u \in \bar{\mathcal{I}}_u.
\label{6.11}
\end{equation}
From the preceding paragraph, for instance, \eqref{6.8} and \eqref{6.9} and the following text, the derivatives are \emph{zero} only when the coefficients $\bar{y}_\emptyset$ and $\bar{C}_{\mathbf{i}_u,\mathbf{p}_u,\boldsymbol{\Xi}_u}^u$  match
the Fourier coefficients $y_\emptyset$ and $C_{\mathbf{i}_u,\mathbf{p}_u,\boldsymbol{\Xi}_u}^u$ of SDD defined in \eqref{6.2} and \eqref{6.3}.  Therefore, the SDD approximation is the best one from the function pool $\mathcal{S}_{\mathbf{p},\boldsymbol{\Xi}}$, as claimed.
\end{proof}

\begin{proposition}
For any $y(\mathbf{X}) \in L^2(\Omega,\mathcal{F},\mathbb{P})$, let
$y_{\mathbf{p},\boldsymbol{\Xi}}(\mathbf{X})$ from \eqref{6.1} be the SDD approximation associated with a chosen degree $\mathbf{p}$ and family of knot sequences $\boldsymbol{\Xi}$.  Then the residual error $y(\mathbf{X})-y_{\mathbf{p},\boldsymbol{\Xi}}(\mathbf{X})$ is orthogonal to the subspace
\begin{equation}
\mathcal{S}_{\mathbf{p},\boldsymbol{\Xi}} =
\boldsymbol{1} \oplus
\displaystyle
\bigoplus_{\emptyset \ne u \subseteq \{1,\ldots,N\}}
\operatorname{span} \left\{ \Psi_{\mathbf{i}_u,\mathbf{p}_u,\boldsymbol{\Xi}_u}^u(\mathbf{X}_u) \right\}_{\mathbf{i}_u \in \bar{\mathcal{I}}_{u,\mathbf{n}_u}}
\subseteq L^2(\Omega,\mathcal{F},\mathbb{P}).
\label{6.12}
\end{equation}
\label{p10}
\end{proposition}

\begin{proof}
Let $g$ described in \eqref{6.10}, with any real-valued coefficients $\bar{y}_\emptyset$ and $\bar{C}_{\mathbf{i}_u,\mathbf{p}_u,\boldsymbol{\Xi}_u}^u$, $\emptyset \ne u\subseteq\{1,\ldots,N\}$, $\mathbf{i}_u \in \bar{\mathcal{I}}_u$, be an arbitrary element of $\mathcal{S}_{\mathbf{p},\boldsymbol{\Xi}}$.  Then
\begin{equation}
\begin{array}{ll}
\displaystyle
  &
\mathbb{E}
\left[  \left\{ y(\mathbf{X}) - y_{\mathbf{p},\boldsymbol{\Xi}}(\mathbf{X}) \right\}
g(\mathbf{X}) \right] \\
= &
\mathbb{E}
\left[
\left\{
y(\mathbf{X}) -
y_\emptyset - \displaystyle \sum_{\emptyset \ne u\subseteq\{1,\ldots,N\}}
\sum_{\mathbf{j}_u \in  \bar{\mathcal{I}}_u}
C_{\mathbf{i}_u,\mathbf{p}_u,\boldsymbol{\Xi}_u}^u \Psi_{\mathbf{j}_u,\mathbf{p}_u,\boldsymbol{\Xi}_u}^u(\mathbf{X}_u)
\right\}
\right.
\\
  &
\times
\left.
\left\{
\displaystyle
\bar{y}_\emptyset + \displaystyle \sum_{\emptyset \ne u\subseteq\{1,\ldots,N\}}
\sum_{\mathbf{i}_u \in  \bar{\mathcal{I}}_u}
\bar{C}_{\mathbf{i}_u,\mathbf{p}_u,\boldsymbol{\Xi}_u}^u \Psi_{\mathbf{i}_u,\mathbf{p}_u,\boldsymbol{\Xi}_u}^u(\mathbf{X}_u)
\right\}
\right]
\\
= &
y_\emptyset \bar{y}_\emptyset +
\displaystyle \sum_{\emptyset \ne u\subseteq\{1,\ldots,N\}}
\sum_{\mathbf{i}_u \in \bar{\mathcal{I}}_u}
C_{\mathbf{i}_u,\mathbf{p}_u,\boldsymbol{\Xi}_u}^u  \bar{C}_{\mathbf{i}_u,\mathbf{p}_u,\boldsymbol{\Xi}_u}^u  -
y_\emptyset \bar{y}_\emptyset -
\displaystyle \sum_{\emptyset \ne u\subseteq\{1,\ldots,N\}}
\sum_{\mathbf{i}_u \in \bar{\mathcal{I}}_u}
C_{\mathbf{i}_u,\mathbf{p}_u,\boldsymbol{\Xi}_u}^u  \bar{C}_{\mathbf{i}_u,\mathbf{p}_u,\boldsymbol{\Xi}_u}^u
\\
= &
0,
\end{array}
\label{6.13}
\end{equation}
where the second equality follows from \eqref{6.2}, \eqref{6.3}, and Proposition \ref{p7}.  Hence, the proposition is proved.
\end{proof}

\subsection{Error bound and convergence}
The results in this section show an error bound describing the approximation power of SDD proposed. However, a theoretical clarification on the boundedness of a linear operator and a formal definition of modulus of smoothness need to be presented first.

\begin{proposition}
The projection operator $P_{\mathcal{S}_{\mathbf{p},\boldsymbol{\Xi}}}:L^2(\mathbb{A}^N,\mathcal{B}^{N},f_{\mathbf{X}}d\mathbf{x}) \to \mathcal{S}_{\mathbf{p},\boldsymbol{\Xi}}$
is a linear, bounded operator.
\label{p11}
\end{proposition}

The proof is omitted here as it is similar to the one presented for the projection operator in the context of SCE \cite{rahman20}. Interested readers should consult the prior work.

\begin{definition}[Schumaker \cite{schumaker07}]
Given a positive integer $\alpha_k \in \mathbb{N}$ and $0 < h_k \le (b_k-a_k)/\alpha_k$, the $\alpha_k$th modulus of smoothness of a function $y \in L^2[a_k,b_k]$ in the $L^2$-norm is defined by
\begin{equation}
\omega_{\alpha_k}(y;h_k)_{L^2[a_k,b_k]} := \sup_{ 0 \le u_k \le h_k}
\left\| \Delta_{u_k}^{\alpha_k} y\right\|_{L^2[a_k,b_k-\alpha_k u_k]},~~h_k > 0,
\label{6.14}
\end{equation}
where
\[
\Delta_{u_k}^{\alpha_k} y(h_k) := \sum_{i=0}^{\alpha_k} (-1)^{\alpha_k-i} \binom{\alpha_k}{i} y(h_k+iu_k)
\]
is the $\alpha_k$th forward difference of $y$ at $h_k$ for any $0 \le u_k \le h_k$.

Moreover, given a multi-index $\boldsymbol{\alpha}=(\alpha_1,\ldots,\alpha_N) \in \mathbb{N}^N$ and any vector $\mathbf{u} \ge \boldsymbol{0}$, let
\[
\Delta_{\mathbf{u}}^{\boldsymbol{\alpha}} = \prod_{k=1}^N \Delta_{u_k}^{\alpha_k}.
\]
Then the $\boldsymbol{\alpha}$-modulus of smoothness of a function $y \in L^2[\mathbb{A}^N]$ in the $L^2$-norm is the function defined by
\begin{equation}
\omega_{\boldsymbol{\alpha}}(y;\mathbf{h})_{L^2[\mathbb{A}^N]} :=
\sup_{ \boldsymbol{0} \le \mathbf{u} \le \mathbf{h}}
\left\| \Delta_{\mathbf{u}}^{\boldsymbol{\alpha}} y \right\|_{L^2[\mathbb{A}_{\boldsymbol{\alpha},\mathbf{u}}^N]},~~
\mathbf{h} > \boldsymbol{0},
\label{6.15}
\end{equation}
where
\[
\mathbb{A}_{\boldsymbol{\alpha},\mathbf{u}}^N = \left\{  \mathbf{x} \in \mathbb{A}^N:
\mathbf{x} + \boldsymbol{\alpha}\otimes\mathbf{u} \in \mathbb{A}^N \right\},~~
\boldsymbol{\alpha}\otimes\mathbf{u} = (\alpha_1 u_1,\ldots,\alpha_N u_N).
\]
\label{d10}
\end{definition}

\begin{lemma}[Rahman \cite{rahman20}]
Let $L^2(\mathbb{A}^N)$ be an unweighted Hilbert space defined as
\begin{equation}
L^2\left( \mathbb{A}^N \right):=
\left\{ y:\mathbb{A}^N \to \mathbb{R}:
\int_{\mathbb{A}^N} |y(\mathbf{x})|^2 d \mathbf{x} < \infty \right\}
\label{6.16}
\end{equation}
with standard norm $\| \cdot \|_{L^2(\mathbb{A}^N)}$.  Then, for any function $y(\mathbf{x}) \in L^2(\mathbb{A}^N,\mathcal{B}^{N},f_{\mathbf{X}}d\mathbf{x})$, it holds that
\begin{equation}
\left\|  y(\mathbf{x}) \right\|_{L^2(\mathbb{A}^N,\mathcal{B}^{N},f_{\mathbf{X}}d\mathbf{x})}  \le
\sqrt{\left\|  f_{\mathbf{X}}(\mathbf{x})\right\|_{L^\infty(\mathbb{A}^N)}}
\left\|  y(\mathbf{x}) \right\|_{L^2(\mathbb{A}^N)},
\label{6.17}
\end{equation}
where $\| \cdot \|_{L^\infty(\mathbb{A}^N)}$ is the infinity norm. Here, it is further assumed that $f_{\mathbf{X}} \in L^\infty(\mathbb{A}^N)$.
\label{l1}
\end{lemma}

\begin{proposition}
For any $y(\mathbf{X}) \in L^2(\Omega, \mathcal{F}, \mathbb{P})$, a sequence of SDD approximations
$\{y_{\mathbf{p},\boldsymbol{\Xi}}(\mathbf{X})\}_{\mathbf{h} > \boldsymbol{0}}$, with $\mathbf{h}=(h_1,\ldots,h_N)$ representing a vector of the largest element sizes, converges to $y(\mathbf{X})$ in mean-square, \emph{i.e.},
\[
\lim_{\mathbf{h} \to \boldsymbol{0}} \mathbb{E}\left[ \left|
y(\mathbf{X})-y_{\mathbf{p},\boldsymbol{\Xi}}(\mathbf{X})\right|^2 \right] = 0.
\]
Furthermore, the sequence of SDD approximations converges in probability, that is, for any $\epsilon>0$,
\[
\lim_{\mathbf{h} \to \boldsymbol{0}} \mathbb{P}\left( \left|y(\mathbf{X})-y_{\mathbf{p},\boldsymbol{\Xi}}(\mathbf{X})\right| > \epsilon \right) = 0;
\]
and converges in distribution, that is, for all points $\xi \in \mathbb{R}$ where $F(\xi)$ is continuous,
\[
\lim_{\mathbf{h} \to \boldsymbol{0}} F_{\mathbf{p},\boldsymbol{\Xi}}(\xi) = F(\xi)
\]
such that $F_{\mathbf{p},\boldsymbol{\Xi}}(\xi):=\mathbb{P}(y_{\mathbf{p},\boldsymbol{\Xi}}(\mathbf{X}) \le \xi)$ and $F(\xi):=\mathbb{P}(y(\mathbf{X}) \le \xi)$ are distribution functions of $y_{\mathbf{p},\boldsymbol{\Xi}}(\mathbf{X})$ and $y(\mathbf{X})$, respectively.  If $F(\xi)$ is continuous on $\mathbb{R}$, then the distribution functions converge uniformly.
\label{p12}
\end{proposition}

\begin{proof}
According to Proposition \ref{p11}, $P_{\mathcal{S}_{\mathbf{p},\boldsymbol{\Xi}}}$ is a linear, bounded operator.  With the tensor modulus of stiffness in mind, use Lemma \ref{l1} and Theorem 12.8 of Schumaker's book \cite{schumaker07} to claim that the $L^2$-error from the SDD approximation is bounded by
\begin{equation}
\left\|  y(\mathbf{x}) - y_{\mathbf{p},\boldsymbol{\Xi}}(\mathbf{x})
\right\|_{L^2(\mathbb{A}^N,\mathcal{B}^{N},f_{\mathbf{X}}d\mathbf{x})}  \le
C \omega_{\mathbf{p}+\boldsymbol{1}}(y;\mathbf{h})_{L^2(\mathbb{A}^N)},
\label{6.16}
\end{equation}
where $C$ is a constant depending only on $\mathbf{p}$, $N$, and $f_{\mathbf{X}}(\mathbf{x})$, and
$\mathbf{p}+\boldsymbol{1} = (p_1+1,\ldots,p_N+1)$.

From Definition \ref{d10}, as $h_k$ approaches \emph{zero}, so does $0 \le u_k \le h_k$.  Taking the limit $u_k \to 0$ inside the integral of the $L^2$ norm, which is permissible for a finite interval and uniformly convergent integrand, the forward difference
\[
\displaystyle
\lim_{u_k \to 0}
\Delta_{u_k}^{\alpha_k} y(x_k) =
y(x_k) \sum_{i=0}^{\alpha_k} (-1)^{\alpha_k-i} \binom{\alpha_k}{i} = 0,
\]
as the sum vanishes for any $\alpha_k \in \mathbb{N}$.  Consequently, the coordinate modulus of smoothness
\[
\omega_{\alpha_k}(y;h_k)_{L^2[a_k,b_k]} \to 0 ~~\text{as}~ h_k \to 0~~\forall \alpha_k \in \mathbb{N}.
\]
Following similar considerations, the tensor modulus of smoothness
\[
\omega_{\boldsymbol{\alpha}}(y;\mathbf{h})_{L^2(\mathbb{A}^N)} \to 0 ~~\text{as}~\mathbf{h} \to \boldsymbol{0}~~\forall \boldsymbol{\alpha} \in \mathbb{N}^N.
\]
Therefore,
\begin{equation}
\displaystyle
\lim_{\mathbf{h} \to \boldsymbol{0}}
\left\|  y(\mathbf{x}) - y_{\mathbf{p},\boldsymbol{\Xi}}(\mathbf{x})
\right\|_{L^2(\mathbb{A}^N,\mathcal{B}^{N},f_{\mathbf{X}}d\mathbf{x})}  = 0,
\label{6.17}
\end{equation}
thus proving the mean-square convergence of $y_{\mathbf{p},\boldsymbol{\Xi}}(\mathbf{X})$ to $y(\mathbf{X})$ for any degree $\mathbf{p} \in \mathbb{N}_0^N$.

In addition, as the SDD approximation converges in mean-square, it does so in probability.  Moreover, as the expansion converges in probability, it also converges in distribution.
\end{proof}

\subsection{Truncation}
According to \eqref{6.1}, the full SDD contains $\prod_{k=1}^N n_k$ basis functions or coefficients.  Therefore, SDD also suffers from the curse of dimensionality if all terms of SDD are retained in the concomitant expansion.  However, in a practical setting, the output function $y(\mathbf{X})$ is likely to have an effective dimension \cite{bellman57} much lower than $N$, meaning that the right side of \eqref{6.1} can be effectively approximated by a sum of lower-dimensional component functions of $y_{\mathbf{p},\boldsymbol{\Xi}}(\mathbf{X})$, but still maintain all random variables $\mathbf{X}$ of a high-dimensional UQ problem.

A straightforward approach to achieving this truncation entails keeping all orthonormal splines in at most $1 \le S \le N$ variables, thereby retaining the degrees of interaction among input variables less than or equal to $S$.  The result is an $S$-variate SDD approximation
\begin{equation}
y_{S,\mathbf{p},\boldsymbol{\Xi}}(\mathbf{X}) :=
\displaystyle
y_\emptyset +
\sum_{\substack{\emptyset \ne u \subseteq \{1,\ldots,N\} \\ 1 \le |u| \le S}}
\sum_{\mathbf{i}_u \in  \bar{\mathcal{I}}_{u,\mathbf{n}_u}}
C_{\mathbf{i}_u,\mathbf{p}_u,\boldsymbol{\Xi}_u}^u \Psi_{\mathbf{i}_u,\mathbf{p}_u,\boldsymbol{\Xi}_u}^u(\mathbf{X}_u)
\label{6.18}
\end{equation}
of $y(\mathbf{X})$, comprising
\begin{equation}
L_{S,\mathbf{p},\boldsymbol{\Xi}} =
\displaystyle
1 +
\sum_{\substack{\emptyset \ne u \subseteq \{1,\ldots,N\} \\ 1 \le |u| \le S}}
\prod_{k \in u} (n_k-1) \le
\prod_{k=1}^N n_k
\label{6.19}
\end{equation}
expansion coefficients including $y_\emptyset$.  Here, the upper bound of $L_{S,\mathbf{p},\boldsymbol{\Xi}}$ kicks in only when $S=N$. In which case, there is no gain in computational efficiency by the SDD approximation.  However, if $S<<N$, as it is anticipated to hold in real-life applications, the number of coefficients in the SDD approximation drops precipitously, ushering in substantial savings of computational effort.

It is important to clarify a few things about the truncated SDD proposed.  First, the right side of \eqref{6.18} contains sums of at most $S$-dimensional orthonormal splines, representing at most $S$-variate SDD component functions of $y_{\mathbf{p},\boldsymbol{\Xi}}(\mathbf{X})$. Therefore, the term ``$S$-variate'' used for the truncated SDD approximation should be interpreted in the context of including at most $S$-degree interaction of input variables, even though $y_{S,\mathbf{p},\boldsymbol{\Xi}}(\mathbf{X})$ is strictly an $N$-variate function.

Second, if $S=N$, then $y_{N,\mathbf{p},\boldsymbol{\Xi}}(\mathbf{X})=y_{\mathbf{p},\boldsymbol{\Xi}}(\mathbf{X})$.  Therefore, the sequence of SDD approximations
$\{y_{S,\mathbf{p},\boldsymbol{\Xi}}(\mathbf{X})\}_{1\le S\le N,~\mathbf{h} > \boldsymbol{0}}$ converges to $y(\mathbf{X})$ in mean-square, that is,
\[
\lim_{S \to N,~\mathbf{h} \to \boldsymbol{0}} \mathbb{E}\left[ \left|
y(\mathbf{X})-y_{S,\mathbf{p},\boldsymbol{\Xi}}(\mathbf{X})\right|^2 \right] = 0.
\]
Furthermore, the sequence of SDD approximations converges in probability, that is, for any $\epsilon>0$,
\[
\lim_{S \to N,~\mathbf{h} \to \boldsymbol{0}} \mathbb{P}\left( \left|y(\mathbf{X})-y_{S,\mathbf{p},\boldsymbol{\Xi}}(\mathbf{X}\right| > \epsilon \right) = 0;
\]
and converges in distribution, that is, for any $\xi \in \mathbb{R}$,
\[
\lim_{S \to N,~\mathbf{h} \to \boldsymbol{0}} F_{S,\mathbf{p},\boldsymbol{\Xi}}(\xi) = F(\xi)
\]
such that $F_{S,\mathbf{p},\boldsymbol{\Xi}}(\xi):=\mathbb{P}(y_{S,\mathbf{p},\boldsymbol{\Xi}}(\mathbf{X}) \le \xi)$ is the distribution function of $y_{S,\mathbf{p},\boldsymbol{\Xi}}(\mathbf{X})$.

Finally, the properties entailing the best approximation and the residual error being orthogonal, as described by Propositions \ref{p9} and \ref{p10}, also apply in the context of the truncated SDD approximation.

\subsection{Computational effort}
Due to identical hierarchical structures of function decompositions, the SDD method has the same order of computational complexity as the existing PDD method.  To expound on SDD's scalability with respect to the problem size or stochastic dimension $N$, consider the univariate ($S=1$) SDD approximation $y_{1,\mathbf{p},\boldsymbol{\Xi}}(\mathbf{X})$ and bivariate ($S=2$) SDD approximation $y_{2,\mathbf{p},\boldsymbol{\Xi}}(\mathbf{X})$, expressed by
\begin{equation}
y_{1,\mathbf{p},\boldsymbol{\Xi}}(\mathbf{X})=
y_\emptyset +
\displaystyle \sum_{k=1}^N \sum_{i_k=2}^{n_k}
C_{i_k,p_k,\boldsymbol{\xi}_k}^k
\psi_{{i_k,p_k,\boldsymbol{\xi}_k}}^k(X_k)
~~\text{and}~~
\label{6.19b}
\end{equation}
\begin{equation}
\begin{array}{rcl}
y_{2,\mathbf{p},\boldsymbol{\Xi}}(\mathbf{X})
&  =  &
y_\emptyset +
\displaystyle \sum_{k=1}^N \sum_{i_k=2}^{n_k}
C_{i_k,p_k,\boldsymbol{\xi}_k}^k
\psi_{{i_k,p_k,\boldsymbol{\xi}_k}}^k(X_k)+
\\
&     & +
\displaystyle \sum_{k_1=1}^{N-1} \sum_{k_2=k_1+1}^{N}
\displaystyle \sum_{i_{k_1}=2}^{n_{k_1}} \sum_{i_{k_2}=2}^{n_{k_2}}
C_{(i_{k_1},i_{k_2}),(p_{k_1},p_{k_2}),\{\boldsymbol{\xi}_{k_1},\boldsymbol{\xi}_{k_2}\}}^{\{k_1,k_2\}}
\psi_{(i_{k_1},i_{k_2}),(p_{k_1},p_{k_2}),\{\boldsymbol{\xi}_{k_1},\boldsymbol{\xi}_{k_2}\}}^{\{k_1,k_2\}}
\left( X_{k_1},X_{k_2} \right),
\end{array}
\label{6.19c}
\end{equation}
respectively. In either approximation, the requisite computational effort can be judged by the associated number of basis functions involved.  For instance, in \eqref{6.19b} and \eqref{6.19c}, there are, respectively,
\[
1+\displaystyle \sum_{k=1}^N (n_k-1)~~\text{and}~~
1+\displaystyle \sum_{k=1}^N (n_k-1)+
\displaystyle \sum_{k_1=1}^{N-1} \sum_{k_2=k_1+1}^{N} (n_{k_1}-1)(n_{k_2}-1)
\]
basis functions. Hence, given the values of $n_k$, $k=1,\ldots,N$, which are decided by $\mathbf{p}$ and $\boldsymbol{\Xi}$, the computational effort with respect to $N$ grows linearly for univariate approximation and quadratically for bivariate approximation.  For example, when $N=15$ and $n_1=\cdots=n_{15}=5$, the univariate and bivariate SDD approximations involve 61 and 1741 basis functions, respectively. In contrast, the numbers of basis functions in the SCE and tensor-product-truncated PCE approximations with five bases in each direction both jump to $5^{15}$, which is significantly greater than that required by either of the two SDD approximations.  In general,  from \eqref{6.19}, the computational effort by an $S$-variate SDD approximation scales $S$-degree-polynomially with respect to $N$.  Therefore, the computational complexity of a truncated SDD is polynomial, as opposed to exponential, thereby deflating the curse of dimensionality to a substantial extent.

\subsection{A few special cases}
The truncated SDD can be viewed as a generalized version subsuming several well-known expansions, namely, SCE, PDD, and PCE, that are commonly used or known in the UQ community.  Three special cases clarify this observation.

\vspace{0.1in}
\begin{enumerate}[{$(1)$}]

\item
Consider the first case where $S=N$.  The resulting SDD approximation is the full SDD, that is, $y_{N,\mathbf{p},\boldsymbol{\Xi}}(\mathbf{X})=y_{\mathbf{p},\boldsymbol{\Xi}}(\mathbf{X})$.  Then the terms of SDD can be reshuffled to be written as
\begin{subequations}
\begin{alignat}{2}
y_{N,\mathbf{p},\boldsymbol{\Xi}}(\mathbf{X}) = & ~y_{\mathbf{p},\boldsymbol{\Xi}}(\mathbf{X}) =
\displaystyle
\sum_{\mathbf{i} \in  \mathcal{I}_{\mathbf{n}}}
C_{\mathbf{i},\mathbf{p},\boldsymbol{\Xi}} \Psi_{\mathbf{i},\mathbf{p},\boldsymbol{\Xi}}(\mathbf{X}),
\label{6.20a}\\
C_{\mathbf{i},\mathbf{p},\boldsymbol{\Xi}} := &
\int_{\mathbb{A}^N} y(\mathbf{x})
\Psi_{\mathbf{i},\mathbf{p},\boldsymbol{\Xi}}(\mathbf{x}) f_{\mathbf{X}}(\mathbf{x})d\mathbf{x},
\label{6.20b}
\end{alignat}
\end{subequations}
involving a complete set of multivariate orthonormal B-splines $\{\Psi_{\mathbf{i},\mathbf{p},\boldsymbol{\Xi}}(\mathbf{X}): \mathbf{i} \in \mathcal{I}_{\mathbf{n}}\}$ in $\mathbf{X}$ and associated coefficients $C_{\mathbf{i}}$, $\mathbf{i} \in \mathcal{I}_{\mathbf{n}}$.  The last expression in \eqref{6.20a} with the coefficients in \eqref{6.20b} represents the concomitant SCE approximation of $y(\mathbf{X})$.  Therefore, when $S=N$, the $\mathbf{p}$th-degree SDD is the same as the $\mathbf{p}$th-degree SCE.
\\

\item
For the second case, let $S<N$.  Given $k=1,\ldots,N$, $0 \le p_k < \infty$, and an interval $[a_k,b_k] \subset \mathbb{R}$, denote by
\begin{equation}
\boldsymbol{\xi}_k^{'}=\{\overset{p_k+1~\mathrm{times}}{\overbrace{a_k,\ldots,a_k}},
\overset{p_k+1~\mathrm{times}}{\overbrace{b_k,\ldots,b_k}}\}
\label{6.21}
\end{equation}
a $(p_k+1)$-open knot sequence with no internal knots, so that  $\boldsymbol{\Xi}^{'}=\{\boldsymbol{\xi}_1^{'},\ldots,\boldsymbol{\xi}_N^{'}\}$.  For the knot sequence $\boldsymbol{\xi}_k^{'}$ in \eqref{6.21}, the resulting B-splines are related to the well-known Bernstein polynomials of degree $p_k$.  Since the set of Bernstein polynomials of degree $p_k$ forms a basis of the polynomial space $\Pi_{p_k}$, the spline space $\mathcal{S}_{k,p_k,\boldsymbol{\xi}_k^{'}}=\Pi_{p_k}$.  Then, going through the same dimensionwise tensor-product construction, it is trivial to show that, indeed, the multivariate spline space $\mathcal{S}_{S,\mathbf{p},\boldsymbol{\Xi}^{'}}$ is spanned by a set of hierarchically ordered multivariate orthonormal polynomials $\{\Psi_{\mathbf{i}_u}^u(\mathbf{x}_u): \emptyset \ne u \subseteq \{1,\ldots,N\},
\boldsymbol{0} \le \mathbf{i}_u \le \mathbf{p}_u \}$ in $\mathbf{x}_u$, resulting in
\begin{subequations}
\begin{alignat}{2}
y_{S,\mathbf{p},\boldsymbol{\Xi}^{'}}(\mathbf{X}) = &
\displaystyle
~y_\emptyset +
\sum_{\substack{\emptyset \ne u \subseteq \{1,\ldots,N\} \\ 1 \le |u| \le S}}~
\sum_{\boldsymbol{0} \le \mathbf{i}_u \le \mathbf{p}_u}
C_{\mathbf{i}_u}^u \Psi_{\mathbf{i}_u}^u(\mathbf{X}_u),
\label{6.22a} \\
C_{\mathbf{i}_u}^u = &
\int_{\mathbb{A}^N} y(\mathbf{x})
\Psi_{\mathbf{i}_u}^u(\mathbf{x}_u) f_{\mathbf{X}}(\mathbf{x})d\mathbf{x}.
\label{6.22b}
\end{alignat}
\end{subequations}
Hence, the $S$-variate SDD approximation in \eqref{6.22a} with coefficients derived from \eqref{6.22b},
associated with a specified degree $\mathbf{p}$, is identical to the $S$-variate PDD approximation with the same degree $\mathbf{p}$.
\\

\item
Finally, set the condition $S=N$ in the second case.  Then the SDD approximation for a specified degree $\mathbf{p}$ becomes
\begin{subequations}
\begin{alignat}{2}
y_{N,\mathbf{p},\boldsymbol{\Xi}^{'}}(\mathbf{X}) = & ~y_{\mathbf{p},\boldsymbol{\Xi}^{'}}(\mathbf{X}) =
\displaystyle
\sum_{\boldsymbol{0} \le \mathbf{i} \le \mathbf{p}}
C_{\mathbf{i}} \Psi_{\mathbf{i}}(\mathbf{X}),
\label{6.23a} \\
C_{\mathbf{i}} = &
\int_{\mathbb{A}^N} y(\mathbf{x})
\Psi_{\mathbf{i}}(\mathbf{x}) f_{\mathbf{X}}(\mathbf{x})d\mathbf{x},
\label{6.23b}
\end{alignat}
\end{subequations}
comprising multivariate orthonormal polynomials $\{\Psi_{\mathbf{i}}(\mathbf{X}):
\boldsymbol{0} \le \mathbf{i} \le \mathbf{p} \}$ in $\mathbf{X}$. Hence, the SDD approximation in \eqref{6.23a} with coefficients attained from \eqref{6.23b} reduces to a $\mathbf{p}$th-degree PCE approximation.
\end{enumerate}

\subsection{Output statistics and other probabilistic characteristics}
The truncated SDD approximation $y_{S,\mathbf{p},\boldsymbol{\Xi}}(\mathbf{X})$ can be viewed as a surrogate of $y(\mathbf{X})$.  Therefore, relevant probabilistic characteristics of $y(\mathbf{X})$, including its first two moments and its probability density function, if it exists, can be estimated from the statistical properties of $y_{S,\mathbf{p},\boldsymbol{\Xi}}(\mathbf{X})$.

Applying the expectation operator on $y_{S,\mathbf{p},\boldsymbol{\Xi}}(\mathbf{X})$ in \eqref{6.18} and imposing Proposition \ref{p7}, its mean
\begin{equation}
\mathbb{E}\left[ y_{S,\mathbf{p},\boldsymbol{\Xi}}(\mathbf{X}) \right] = y_\emptyset =
\mathbb{E}\left[ y(\mathbf{X})\right]
\label{6.24}
\end{equation}
is independent of $S$, $\mathbf{p}$, and $\boldsymbol{\Xi}$. More importantly, the SDD approximation always yields the exact mean, provided that the expansion coefficient $y_\emptyset$ is calculated exactly.

Applying the expectation operator again, this time on $[y_{S,\mathbf{p},\boldsymbol{\Xi}}(\mathbf{X})- y_\emptyset]^2$, and employing Proposition \ref{p7} one more time results in the variance
\begin{equation}
\operatorname{var}\left[ y_{S,\mathbf{p},\boldsymbol{\Xi}}(\mathbf{X}) \right] =
\sum_{\substack{\emptyset \ne u \subseteq \{1,\ldots,N\} \\ 1 \le |u| \le S}}
\sum_{\mathbf{i}_u \in  \bar{\mathcal{I}}_{u,\mathbf{n}_u}}
{C_{\mathbf{i}_u,\mathbf{p}_u,\boldsymbol{\Xi}_u}^{u^2}}
\le \operatorname{var}[y(\mathbf{X})]
\label{6.25}
\end{equation}
of $y_{S,\mathbf{p},\boldsymbol{\Xi}}(\mathbf{X})$.  Therefore, the second-moment properties of an SDD approximation in \eqref{6.24} and \eqref{6.25} are solely determined by an appropriately truncated set of expansion coefficients.  The formulae for the mean and variance of the SDD approximation are similar to those reported for the PDD approximation, although the respective expansion coefficients involved are not.  The primary reason for this similarity stems from the use of a hierarchically ordered orthonormal basis in both expansions.

Being convergent in probability and in distribution, the probability distribution function and density function of $y(\mathbf{X})$, if it exists, can also be estimated by Monte Carlo simulation (MCS) of the SDD approximation $y_{S,\mathbf{p},\boldsymbol{\Xi}}(\mathbf{X})$.  With the expansion coefficients calibrated, the simulation of $y_{S,\mathbf{p},\boldsymbol{\Xi}}(\mathbf{X})$ entails inexpensive evaluations of simple spline functions.


\section{Numerical experiments}
Two examples involving two- and five-dimensional output functions of uniformly distributed random variables are presented. The objective is to evaluate the approximation power of SDD in terms of the second-moment statistics and probability distribution of $y(\mathbf{X})$ and compare the new SDD results with those obtained from PCE, PDD, and sparse-grid quadrature.

The coordinate degrees for the SDD approximation are identical, that is, $p_1=\cdots=p_N=p$ (say).  So are the knot sequences for SDD, that is, $\boldsymbol{\xi}_1=\cdots=\boldsymbol{\xi}_N=\boldsymbol{\xi}$ (say) with a mesh comprising the largest element sizes $h_1=\cdots=h_N=h$.  The degree $p$ and/or mesh size $h$ were varied as desired.  The basis for a $p$th-degree PCE or PDD was obtained from an appropriate set of Legendre orthonormal polynomials in input variables, whereas the basis for an SDD, given a degree $p$ and a knot sequence of mesh size $h$, was generated from the Cholesky factorization of the spline moment matrix.  Given the uniform distribution, the spline moment matrix was constructed analytically.  All knot sequences are $(p+1)$-open and consist of uniformly spaced distinct knots with even numbers of elements.  All expansion coefficients in Example 1 were calculated exactly, whereas in Example 2, the coefficients were estimated by least-squares regression. The relative error in Example 1 is defined as the absolute difference between the exact and approximate variances, divided by the exact variance.

\subsection{Example 1: a nonsmooth function}
Defined on the square $\mathbb{A}^2=[-1,1]^2$, consider a nonsmooth function of two uniformly distributed random variables $X_1$ and $X_2$, each distributed over $[-1,1]$:
\begin{equation}
y(X_1,X_2) = g(X_1)+g(X_2) + \frac{1}{5} g(X_1)g(X_2),
\label{7.1}
\end{equation}
where, for $i=1,2$,
\begin{equation}
g(x_i) =
\begin{cases}
1,              & -1 \le x_i \le 0, \\
\exp(-10 x_i),  &  0 <   x_i \le 1.
\end{cases}
\label{7.2}
\end{equation}
A graph of the function in Figure \ref{fig1}(a) indicates that $y$ has a flat region on $[-1,0]^2$, and then it falls off exponentially on both sides.  Clearly, the function is continuous, but it has discontinuous partial derivatives across the lines $x_1=0$ and $x_2=0$.  Such functions are difficult to approximate by polynomials.

Figures \ref{fig2}(b) through  \ref{fig2}(f) present graphs of several approximations of $y(x_1,x_2)$ from PCE and SDD.  The second-order PCE approximation in Figure \ref{fig2}(b) commits a relative variance error
of $0.178781$, and is inadequate. The 20th-order PCE approximation in Figures \ref{fig2}(c) show improvements by reducing the errors to $2.19198\times 10^{-3}$, but not to a magnitude expected from such an impractically high expansion order.

In contrast, the bivariate, linear ($S=2$, $p=1$) and bivariate, quadratic ($S=2$, $p=2$) SDD approximations in Figures \ref{fig1}(d) and \ref{fig1}(e), obtained for a mesh size of $h=1/10$ ($I=20$), match the exact function well, producing respective variance errors of $2.88408\times 10^{-4}$ and $1.28264\times 10^{-3}$.  Clearly, the approximation quality of SDD, whether linear or quadratic, surpasses that of the 20th-degree PDD approximations, yet the quality in Figure \ref{fig2}(e) is inferior to that in \ref{fig2}(d).  This apparent anomaly of a linear SDD approximation producing results cut above a quadratic SDD approximation can be explained by examining the knot sequences used.  Due to even numbers of elements, there exists a central knot in each coordinate direction for the cases of $p=1$ and $p=2$.  However, for $p=2$, the first-order derivatives are continuous across the central knot in both directions.  This is why the quadratic SDD approximations are smoother than the linear SDD approximations or the exact function.  In this case, as $y(x_1,x_2)$ is not differentiable at the central knot, the linear approximation outperforms the quadratic approximation.  Furthermore, if the central knot is repeated (multiplicity of two) in the knot sequences, the quadratic SDD approximation, exhibited in Figure1 \ref{fig2}(f), is even better than the linear SDD approximations, resulting in the variance error of $3.31017\times 10^{-6}$. Any distinction between the quadratic SDD approximation in \ref{fig2}(f) and the exact function $y$ in \ref{fig1}(a) is impalpable to the naked eye.

\begin{figure}[htbp!]
\begin{centering}
\includegraphics[scale=0.71]{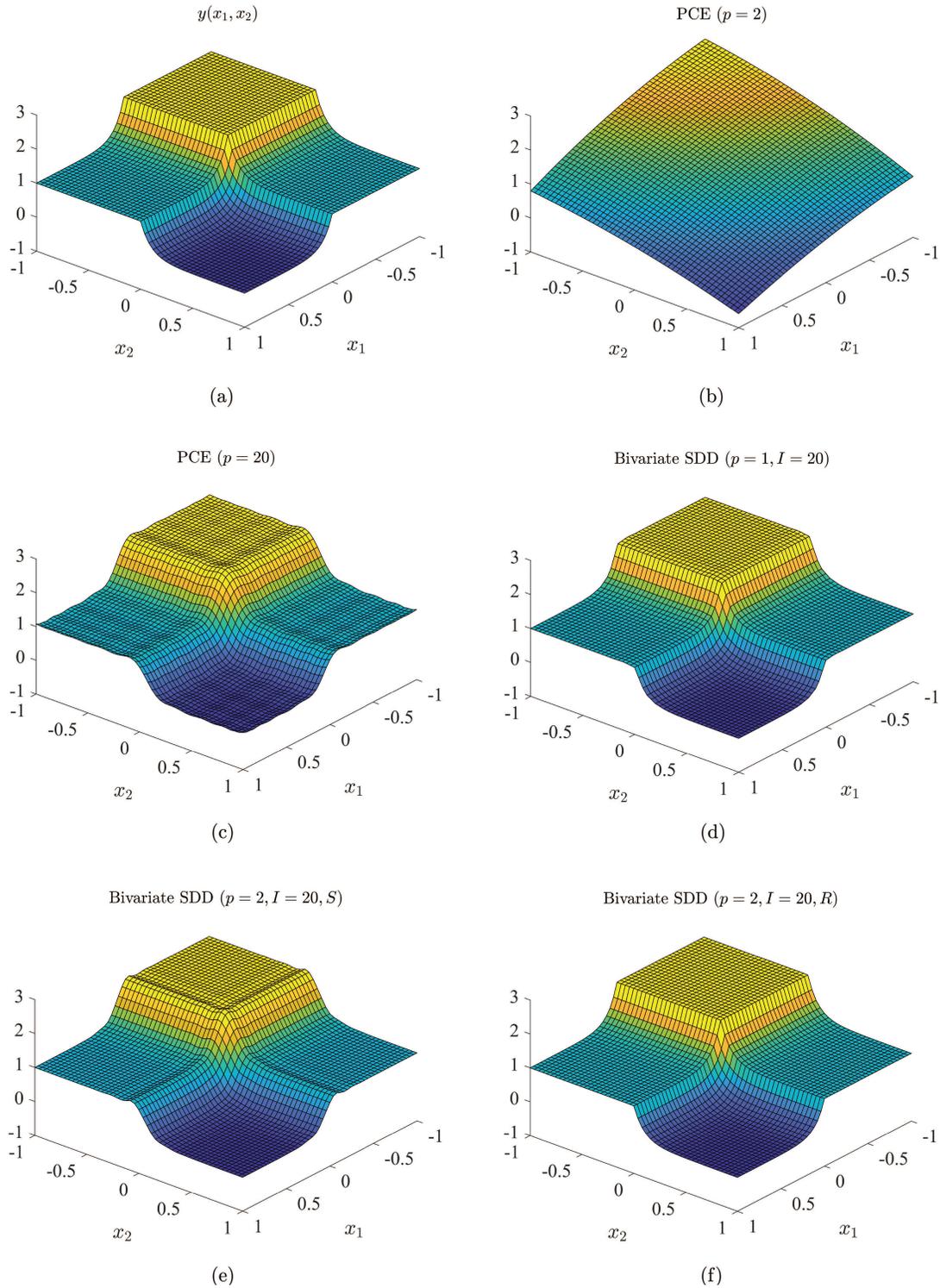}
\par\end{centering}
\caption{Graphs of functions in Example 1;
(a) exact;
(b) second-order PCE;
(c) 20th-order PCE;
(d) bivariate, linear SDD;
(e) bivariate, quadratic SDD with simple (``$S$'') knots;
(f) bivariate, quadratic SDD with repeated (``$R$'') central knots.}
\label{fig2}
\end{figure}

Although SDD enables a greater flexibility than PCE in exploiting low expansion orders, a comparison between PCE and SDD approximations, including the well-studied sparse-grid quadrature, pertaining to their computational efforts is warranted.  Theoretically, such a comparison can be made by examining the total numbers of requisite basis functions from these methods.  This was achieved by assessing (1) PCE approximations for ten distinct values of $p=1,2,4,6,8,10,12,14,16,20$; (2) bivariate SDD approximations for two distinct values of $p=1,2$, and ten distinct mesh sizes of $h=2,1,1/2,1/3,1/4,1/5,1/6,1/7,1/8,1/10$; and (3) sparse grids with the Clenshaw-Curtis quadrature rule for levels varying from one through seven. Note that the number of bases for the sparse-grid method is equivalent to the number of integration points.

Figure \ref{fig3} depicts how the relative error in variance, calculated by various methods, decays against the number of basis functions. From that figure, the sparse grids and PCE methods struggle to provide results as accurate as those obtained by SDD methods of order ($p$) only up to two. This is largely due to the nonsmoothness in the original function $y$.  As explained earlier, the errors are larger for quadratic SDD than for linear SDD, when simple knots are used.  If, however, repeated knots are placed, then the quadratic SDD method becomes markedly more accurate than any other methods. Moreover, the convergence is steeper for quadratic SDD with repeated knots than linear SDD. Overall, the numerical evidence from this example reveals significant contribution of the proposed SDD method in terms of both efficiency and accuracy, while the rate of convergence is also substantially higher for functions with harsh regularities.

\begin{figure}[htbp!]
\begin{centering}
\includegraphics[scale=0.6]{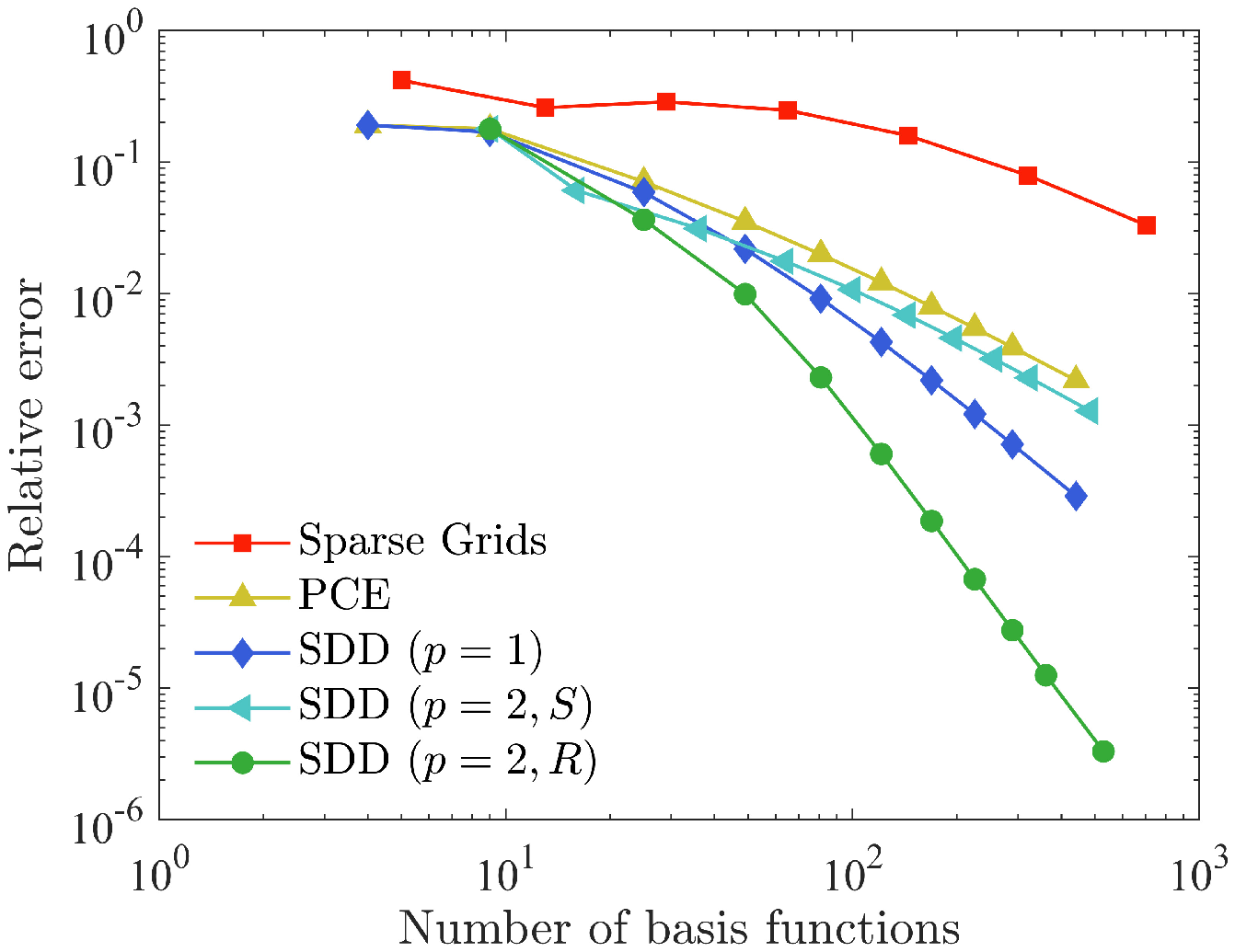}
\par\end{centering}
\caption{Relative errors in the variances from PCE, sparse-grid, and SDD approximations in Example 1. Note: the relative error is defined as the absolute difference between the exact and approximate variances, divided by the exact variance.}
\label{fig3}
\end{figure}

\subsection{Example 2: a cylinder under a pair of concentrated forces}
The second example is concerned with solving a stochastic partial differential equation (PDE) commonly encountered in stochastic mechanics of solids.  A 4-unit-long ($L=4$) open cylinder with both ends fixed in the longitudinal direction ($z$-axis) is subjected to a pair of collinear, concentrated pulling or pushing forces $F$ in the $y$ direction, as shown in Figure \ref{fig4}(a). A positive or negative value of $F$ represents pulling or pushing forces, respectively. There are five independent and uniformly distributed random variables in this problem: mean radius $R$, thickness $t$, Young's modulus $E$, Poisson's ratio $\nu$, and concentrated force $F$, with their statistical properties described in Table \ref{table1}.  Given the input random vector $\mathbf{X}=(R,t,E,\nu,F)$, the underlying stochastic PDE calls for finding the displacement $\mathbf{u}(\mathbf{z};\mathbf{X})$ and stress $\boldsymbol{\sigma}(\mathbf{z};\mathbf{X})$ solutions at a spatial coordinate $\mathbf{z}=(z_1,z_2,z_3) \in \mathcal{D} \subset \mathbb{R}^3$, satisfying $P$-almost surely
\begin{equation}
\begin{array}{rcl}
\mathbf{\nabla} \cdot \boldsymbol{\sigma}(\mathbf{z};\mathbf{X})
&  =  &  \boldsymbol{0}~\text{in}~\mathcal{D} \subset \mathbb{R}^3, \\
\boldsymbol{\sigma}(\mathbf{z};\mathbf{X}) \cdot \mathbf{n}(\mathbf{z};\mathbf{X}),
&  =  &  \bar{\mathbf{t}}(\mathbf{z};\mathbf{X})~\text{on}~\partial \mathcal{D}_t, \\
\mathbf{u}(\mathbf{z};\mathbf{X})
&  =  &  \bar{\mathbf{u}}(\mathbf{z};\mathbf{X})~\text{on}~\partial \mathcal{D}_u,
\end{array}
\label{7.5}
\end{equation}
such that
$
\partial \mathcal{D}_t \cup \partial \mathcal{D}_u = \partial \mathcal{D},~
\partial \mathcal{D}_t \cap \partial \mathcal{D}_u = \emptyset.
$
Here, $\mathbf{\nabla}:=(\partial/\partial z_1,\partial/\partial z_2,\partial/\partial z_3)$,
$\bar{\mathbf{t}}(\mathbf{z};\mathbf{X})$ is prescribed traction on $\partial \mathcal{D}_t$,
$\bar{\mathbf{u}}(\mathbf{z};\mathbf{X})$ is prescribed displacement on $\partial \mathcal{D}_u$, and
$\mathbf{n}(\mathbf{z};\mathbf{X})$ is unit outward normal vector.  A finite-element analysis (FEA) model, comprising 1152 eight-noded, linear, hexahedral elements (two elements in the radial direction, 48 elements in the circumferential direction, and 12 elements in the longitudinal direction) illustrated in Figure \ref{fig4}(b), was developed to solve the associated Galerkin weak form of \eqref{7.5}.  The objective of this example is to estimate, using the FEA computational model, various probabilistic characteristics of the displacement field of the cylinder.

\begin{figure}[htbp]
\begin{centering}
\includegraphics[scale=0.72, clip=true]{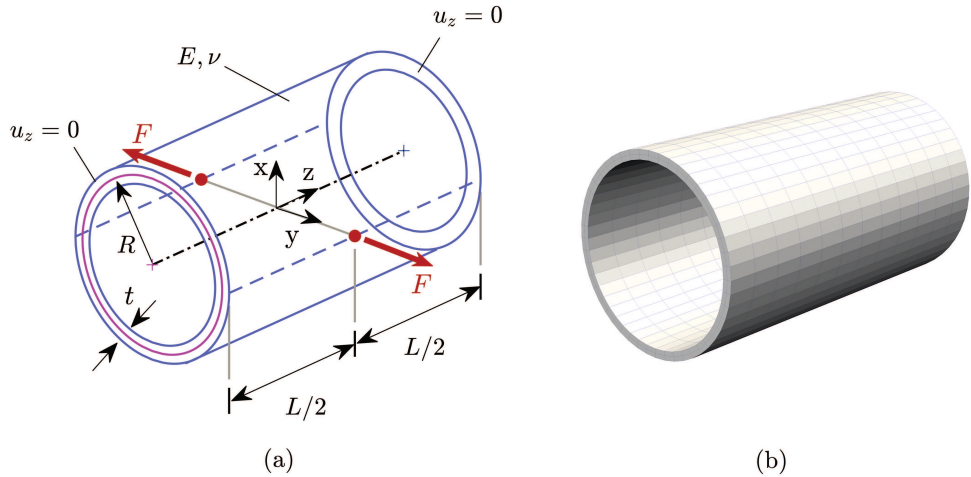}
\par\end{centering}
\caption{A cylinder subject to a pair of pulling or pushing forces;  (a) geometry and loads;
(b) FEA model.}
\label{fig4}
\end{figure}

\begin{table}[htbp]
    \begin{minipage}{.5\linewidth}
\caption{Statistical properties of random input in Example 2.}
  \begin{tabular}{cccc}
\hline
{\footnotesize{}$\begin{array}{c}
\mathrm{Random}\\
\mathrm{variable}
\end{array}$} & {\footnotesize{}Mean} & {\footnotesize{}$\begin{array}{c}
\mathrm{St.}\\
\mathrm{dev.}
\end{array}$} & {\footnotesize{}$\begin{array}{c}
\mathrm{Bounds\,of}\\
\mathrm{distribution}
\end{array}$}\tabularnewline
\hline
{\footnotesize{}$R$} & {\footnotesize{}1} & {\footnotesize{}0.0121} & {\footnotesize{}$[0.979,1.021]$}\tabularnewline
{\footnotesize{}$t$} & {\footnotesize{}0.1} & {\footnotesize{}0.0012} & {\footnotesize{}$[0.098,0.102]$}\tabularnewline
{\footnotesize{}$E$} & {\footnotesize{}1} & {\footnotesize{}0.0577} & {\footnotesize{}$[0.9,1.1]$}\tabularnewline
{\footnotesize{}$\nu$} & {\footnotesize{} 1/3} & {\footnotesize{}0.0096} & {\footnotesize{}$[0.95/3,1.05/3]$}\tabularnewline
{\footnotesize{}$F^{(\mathrm{a})}$} & {\footnotesize{}-0.0005} & {\footnotesize{}0.00057735} & {\footnotesize{}$[-0.0015,0.0005]$}\tabularnewline
\hline
\end{tabular}{\footnotesize\par}
{\scriptsize{}(a)
Positive $F$ = pulling force; negative $F$ = pushing forces.}{\scriptsize\par}
\label{table1}
    \end{minipage}%
    \begin{minipage}{.5\linewidth}
      \centering
\caption{Statistical properties of $\Delta u_{y,F}$ by various methods in Example 2.}
\begin{tabular}{ccc}
\hline
{\footnotesize{}Method} & {\footnotesize{}Mean} & {\footnotesize{}St. dev.}\tabularnewline
\hline
{\footnotesize{}Bivariate, first-order PDD} & {\footnotesize{}0.217536503} & {\footnotesize{}0.135859716}\tabularnewline
{\footnotesize{}Bivariate, fifth-order PDD} & {\footnotesize{}0.215657135} & {\footnotesize{}0.152458926}\tabularnewline
{\footnotesize{}Bivariate, first-order SDD} & {\footnotesize{}0.215667677} & {\footnotesize{}0.152458835}\tabularnewline
{\footnotesize{}Crude MCS (50,000 samples)} & {\footnotesize{}0.215366279} & {\footnotesize{}0.151960153}\tabularnewline
\hline
\end{tabular}
\label{table2}
    \end{minipage}
\end{table}

Two PDD and one SDD methods were employed to calculate the second-moment statistics of displacements: (1) the bivariate, first-order ($S=2$, $p=1$) PDD; (2) the bivariate, fifth-order ($S=2$, $p=5$) PDD; and (3) the bivariate, linear ($S=2$, $p=1$) SDD with the eight subintervals in each coordinate direction ($I=8$). All orthogonal polynomials or splines involved were determined analytically.  However, unlike in Example 1, the expansion coefficients of PDD or SDD here cannot be calculated exactly.  Alternatively, a linear regression analysis was performed to estimate the coefficients in two steps:  (1) draw $L_r \in \mathbb{N}$ samples of random input $\mathbf{X}$ from their probability distributions and hence calculate for each input sample the corresponding output sample of displacement from the FEA model and (2) conduct a least-squares best fit of the PDD or SDD approximation to $L_r$ pairs of input-output samples.  Therefore, the PDD and SDD approximations in this example contain not only projection errors, but also numerical errors due to regression.  Nonetheless, all PDD and SDD approximations require the same computational effort and are proportional to $L_r$, as all sample calculations (FEA) require practically the same effort.  A sample size of $L_r=3000$ was chosen for regression and deemed adequate for this problem.

\subsubsection{Second-moment characteristics}
The behavior of a pulled or pushed cylinder is often evaluated by examining the relative radial displacement of the two load points.  Denote by $\Delta u_{y,F}$ the absolute value of such relative displacement.  Here, $\Delta u_{y,F}$ is an output random variable of interest that depends on the random input $\mathbf{X}$.  More importantly, since $F$ takes on positive (pull) and negative (push) values, $\Delta u_{y,F}$ is a nonsmooth function of input variables.  Table \ref{table1} lists the mean and standard deviation of $\Delta u_{y,F}$ calculated by the three aforementioned methods and a crude MCS.  Due to the computational expense of FEA, the MCS was conducted for a sample size of 50,000, which should be adequate for providing benchmark solutions of the second-moment characteristics.  The agreement between the means by PDD or SDD approximations and MCS in Table \ref{table2} is very good. However, the SDD approximation is more accurate than the first-order PDD approximation when estimating the standard deviation.  The PDD method becomes competitive only when used in conjunction with the fifth-order approximation.  This is attributed to the nonsmooth response behavior where a low-order SDD approximation can be adapted to produce more accurate results than those possible by a low-order PDD approximation.

A similar second-moment analysis was performed for the root-mean-square (RMS) values ($\sqrt{u_x^2+u_y^2+u_z^2}$) of three displacement components ($u_x, u_y, u_z$) at all finite-element nodes.  Figures \ref{fig5}(a) through \ref{fig5}(d) portray  contour plots of the standard deviations of RMS displacements obtained from MCS and three approximate methods defined in the preceding.  When comparing the PDD results in Figures \ref{fig5}(b) and \ref{fig5}(c) with the MCS-generated solution in Figure \ref{fig5}(a), the first-order approximation is unsatisfactory, although the fifth-order approximation produces the standard deviations well.  Indeed, there are discernable differences in the contour plots obtained for the first-order PDD approximation and MCS solution. In contrast, the standard deviations by the bivariate, linear SDD approximation in Figure \ref{fig5}(d) are remarkably close to the MCS result in Figure \ref{fig5}(a).  Therefore, a first-order SDD approximation with only eight subintervals ($I=8$) demonstrates the superiority of SDD over PDD approximations.

\begin{figure}[htbp]
\begin{centering}
\includegraphics[scale=0.7,clip=true]{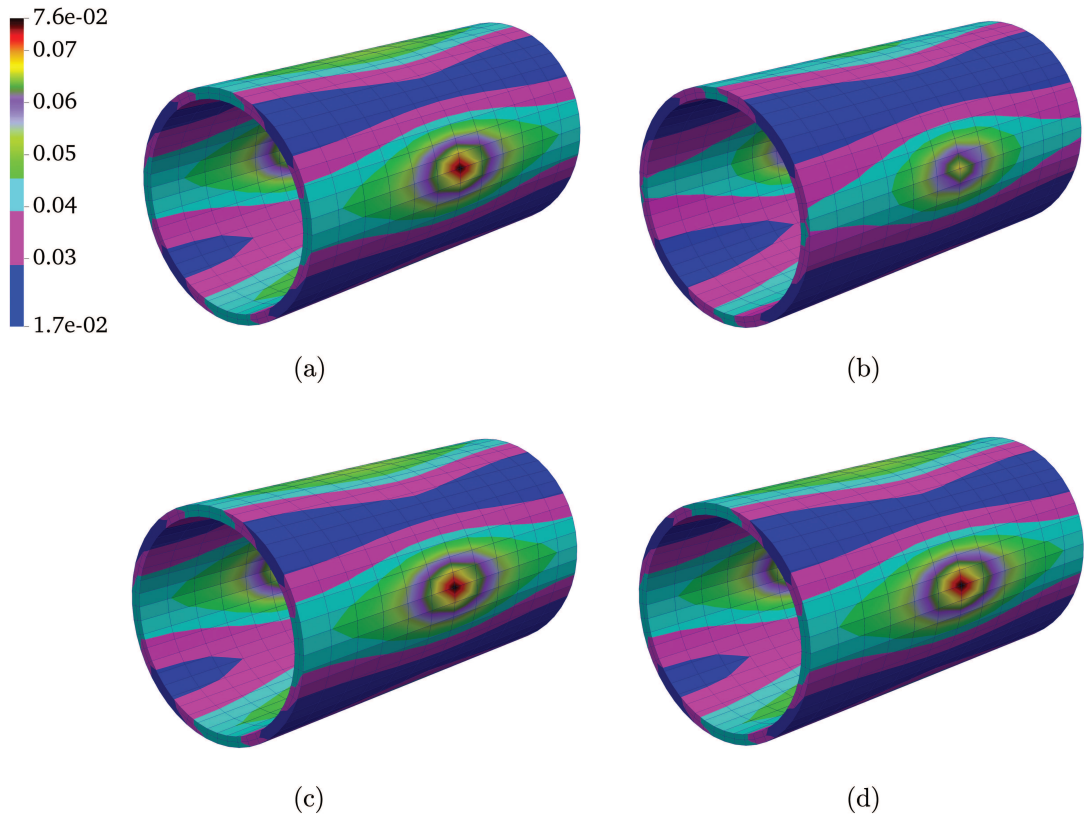}
\par\end{centering}
\caption{Contour plots of standard deviations of RMS displacements of the cylinder by various methods in Example 2;
(a) crude MCS (50,000 samples);
(b) bivariate, first-order ($S=2$, $p=1$) PDD;
(c) bivariate, fifth-order ($S=2$, $p=5$) PDD;
(d) bivariate, linear ($S=2$, $p=1$, $I=8$) SDD.}
\label{fig5}
\end{figure}

\subsubsection{Probability distribution}
Figure \ref{fig6} illustrates the distribution functions of $\Delta u_{y,F}$ obtained by the aforementioned methods, including MCS. The same 50,000 samples generated for verifying the statistics in Table \ref{table2} were utilized for MCS estimates in Figure 5. However, since a PDD or SDD yields an explicit approximation of $\Delta u_{y,F}$ in terms of multivariate polynomials or splines, a relatively large sample size -- a million in this example -- was selected to sample the approximation for estimating the corresponding distribution function.  According to Figure \ref{fig6}, the distribution functions by the linear SDD approximation and MCS match extremely well over the entire support of $\Delta u_{y,F}$.  In contrast, the first- and fifth-order PDD approximations produce satisfactory estimates of the distribution function only around the mean; however, in the tail (lower) region, there are significant discrepancies.  Additionally, the distribution function obtained from a tenth-order approximation, similarly depicted in Figure \ref{fig6}, does not bring a tangible improvement to PDD.  It is interesting to note that a fifth-order PDD approximation, while it produces satisfactory second-moment properties, may not necessarily accurately capture tails of distribution functions.  This is because higher-order moments are involved when estimating distribution functions.

\begin{figure}[htbp]
\begin{centering}
\includegraphics[scale=0.45,clip=true]{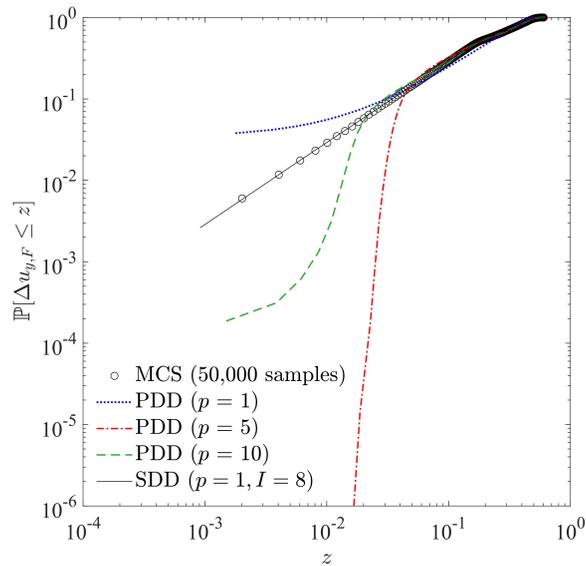}
\par\end{centering}
\caption{Cumulative probability distribution function of $\Delta u_{y,F}$ by various methods in Example 2.}
\label{fig6}
\end{figure}

It is important to clarify that the probabilistic results presented in this example should be viewed in the context of a fixed FEA discretization of the cylinder.  Moreover, as no meaningful differences were found in the results of trivariate PDD or SDD approximations, their results are not reported here.

\section{Application}
This section illustrates the effectiveness of the proposed SDD method in solving a large-scale practical engineering problem. It involves predicting the dynamic behavior of a sport utility vehicle (SUV) in terms of the statistical properties of natural frequencies and mode shapes.

\subsection{An SUV body-in-white model}
An SUV body-in-white (BIW) model, referring to the automotive design stage where only a car body's sheet metal components have been welded together, is presented in Figure \ref{fig7}(a).  It consists of the bare metal shell of the frame body, including fixed windshields. An FEA mesh comprising 127,213 linear shell elements and 794,292 active degrees of freedom is depicted in \ref{fig7}(b).  As displayed in Figure \ref{fig7}(a), the BIW model contains 17 distinct materials having random properties, including 17 Young's moduli and 17 mass densities. In total, there are 34 random variables as follows: $X_{1}$ to $X_{17}$ = Young's moduli of materials 1 to 17; and $X_{18}$ to $X_{34}$ = mass densities of materials 1 to 17. Their means, $\mu_{i}:=\mathbb{E}\left[X_{i}\right]$, $i=1,\ldots,34$, are listed in Table \ref{table3}. Each variable follows an independent, truncated Gaussian distribution with lower limit $a_{i}=0.8\mu_{i}$, upper limit $b_{i}=1.2\mu_{i}$, and coefficient of variation $v_{i}=0.1$. The deterministic Poisson's ratios are as follows: 0.28 for materials 1 through 13; 0.2 for materials 14 and 15; and 0.3 for materials 16 and 17.  This is an undamped system. Therefore, the damping factors for all materials are equal to \emph{zero}.
\begin{figure}[htbp]
\begin{centering}
\includegraphics[scale=1]{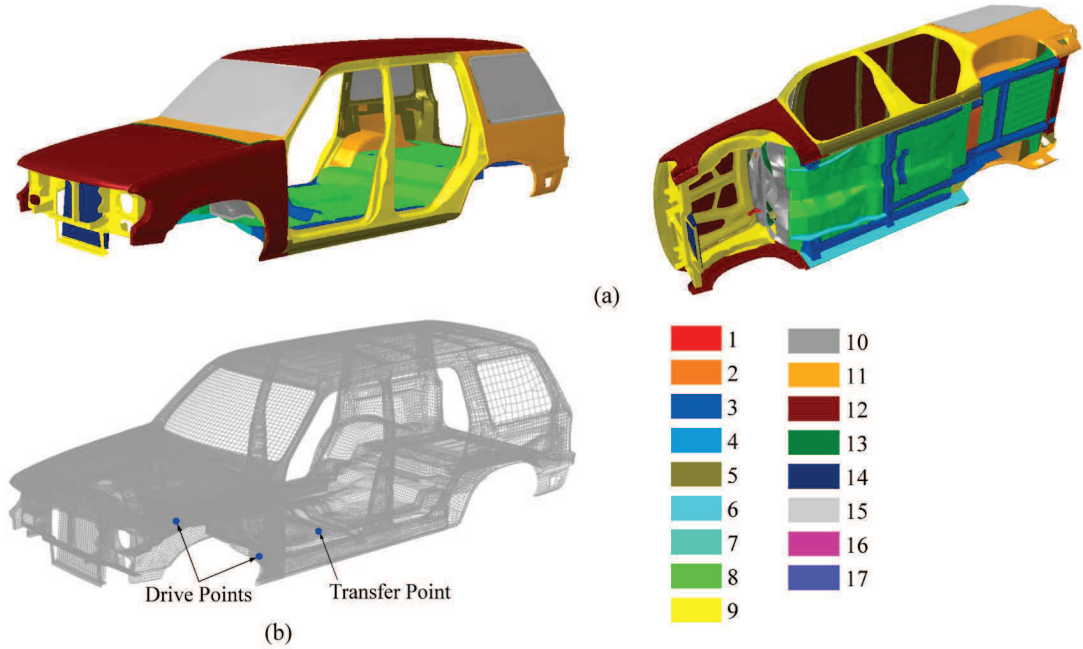}
\par\end{centering}
\caption{Steady-state dynamic analysis of an SUV \cite{yadav13}: (a) BIW model; (b) FEA mesh.}
\label{fig7}
\end{figure}

\begin{table}[htbp]
\footnotesize
\caption{Mean values of random input variables for the SUV BIW model.}
\begin{centering}
\begin{tabular}{ccccccc}
\hline
         & Young's    & Mass density,           & &           & Young's    & Mass density,       \tabularnewline
Material & mod.,\,GPa & kg/$\mathrm{m}^{3}$     & &  Material & mod.,\,GPa & kg/$\mathrm{m}^{3}$ \tabularnewline
\hline
1 & 207 & 9500   & & 10 & 207 & 29,260  \tabularnewline
2 & 207 & 9500   & & 11 & 207 & 30,930  \tabularnewline
3 & 207 & 8100   & & 12 & 207 & 37,120  \tabularnewline
4 & 207 & 29,260 & & 13 & 207 & 52,010  \tabularnewline
5 & 207 & 29,260 & & 14 & 69  & 2700    \tabularnewline
6 & 207 & 37,120 & & 15 & 69  & 2700    \tabularnewline
7 & 207 & 9500   & & 16 & 20  & 1189    \tabularnewline
8 & 207 & 8100   & & 17 & 200 & 1189    \tabularnewline
9 & 207 & 8100   & &    &     &         \tabularnewline
\hline
\end{tabular}
\par\end{centering}
\label{table3}
\end{table}

\begin{table}[htbp]
\footnotesize
\caption{Second-moment statistics of first ten natural frequencies of the BIW model by various methods.}
\begin{centering}
\begin{tabular}{cccccccccc}
\hline
 & \multicolumn{6}{c}{$\mathrm{Univariate\,SDD\,methods\,(data\,size=500\,samples)}$} &  & \multicolumn{2}{c}{}\tabularnewline
\cline{2-7}
 & \multicolumn{2}{c}{$p=1,I=2$} & \multicolumn{2}{c}{$p=1,I=4$} & \multicolumn{2}{c}{$p=2,I=4$} &  & \multicolumn{2}{c}{$\begin{array}{c}
\mathrm{Crude\,MCS}\\
\mathrm{(5000\,samples)}
\end{array}$}\tabularnewline
\cline{2-7} \cline{9-10}
 & Mean & St. dev.  & Mean & St. dev. & Mean & St. dev. &  & Mean & St. dev.\tabularnewline
Mode & (Hz) & (Hz) & (Hz) & (Hz) & (Hz) & (Hz) &  & (Hz) & (Hz)\tabularnewline
\hline
1 & 2.2964 & 0.1422 & 2.2965 & 0.1425 & 2.2965 & 0.1426 &  & 2.2970 & 0.1438\tabularnewline
2 & 6.3203 & 0.3784 & 6.3210 & 0.3785 & 6.3207 & 0.3784 &  & 6.3189 & 0.3795\tabularnewline
3 & 8.2239 & 0.4158 & 8.2186 & 0.4234 & 8.2209 & 0.4358 &  & 8.2230 & 0.4360\tabularnewline
4 & 8.3987 & 0.5015 & 8.3985 & 0.5025 & 8.3981 & 0.5033 &  & 8.3911 & 0.4989\tabularnewline
5 & 8.9711 & 0.4170 & 8.9781 & 0.4367 & 8.9762 & 0.4458 &  & 8.9657 & 0.4787\tabularnewline
6 & 10.9581 & 0.5235 & 10.9574 & 0.5247 & 10.9568 & 0.5267 &  & 10.9570 & 0.5359\tabularnewline
7 & 13.2654 & 0.7706 & 13.2654 & 0.7728 & 13.2657 & 0.7770 &  & 13.2554 & 0.7681\tabularnewline
8 & 13.3114 & 0.8055 & 13.3109 & 0.8064 & 13.3124 & 0.8080 &  & 13.3053 & 0.8051\tabularnewline
9 & 14.8789 & 0.6353 & 14.8867 & 0.6693 & 14.8842 & 0.6816 &  & 14.8899 & 0.7009\tabularnewline
10 & 15.7455 & 0.6074 & 15.7432 & 0.6210 & 15.7349 & 0.6308 &  & 15.7504 & 0.6385\tabularnewline
\hline
\end{tabular}
{\footnotesize\par}
\par\end{centering}
\label{table4}
\end{table}

\subsection{Modal dynamics and results}
A mode-based steady-state dynamic analysis was performed to obtain eigensolutions of the BIW model.  Such analysis is often required to determine the frequency response functions, \emph{e.g.}, receptance, mobility, and inertance, resulting in a vehicle body design with desired dynamic characteristics \cite{yadav13}. The Lanczos method \cite{lanczos50} available in ABAQUS (Version 2019) \cite{abaqus2019} was employed for calculating several natural frequencies and mode shapes.

Due to uncertainty in material properties, the eigensolutions of the SUV BIW model are stochastic. Three univariate SDD methods, two with linear approximations ($p=1$, $I=2,4$) and one with a quadratic approximation ($p=2$, $I=4$), and crude MCS entailing 5000 samples were employed to find their second-moment properties. The expansion coefficients of SDD were again estimated by least-squares regression from a MCS-generated experimental design consisting of only 500 samples (FEA).

Table \ref{table4} lists the means and standard deviations of natural frequencies of the BIW model associated with the first ten non-rigid-body modes, obtained using the aforementioned SDD methods and MCS.  A comparison between respective statistics of frequencies by these methods point to a very good accuracy of all three SDD methods.  For instance, the errors in the standard deviations of the tenth natural frequency calculated by the three SDD methods when compared with MCS are $4.9\%$, $2.7\%$ and $1.2\%$, respectively.  The precision of the quadratic SDD approximation with four subintervals is especially high, yet it still requires only 500 samples (FEA).

In addition, the $L^2$-norm, that is, the square root of sum of squares, of nodal displacements at each node of the BIW model representing the magnitude of mode shape was calculated. Figures \ref{fig8}(a) through \ref{fig8}(d) portray contour plots of standard deviations of the magnitude of an arbitrarily selected 15th mode shape (non-rigid-body), calculated using crude MCS and three SDD methods.  Similar results can be generated for other mode shapes if desired. Again, the SDD-generated mode shapes in Figures \ref{fig8}(b) through \ref{fig8}(d) are remarkably close to those generated by crude MCS in Figure \ref{fig8}(a).  For both natural frequencies and mode shapes, the SDD methods deliver very accurate second-moment statistics, asking for only ten percent of the computational effort mandated by crude MCS.  Indeed, the success of conducting UQ analysis for the SUV BIW model comprising 34 random variables demonstrates the viability of SDD in solving industrial-scale engineering problems.
\begin{figure}[htbp]
\begin{centering}
\includegraphics[scale=0.83]{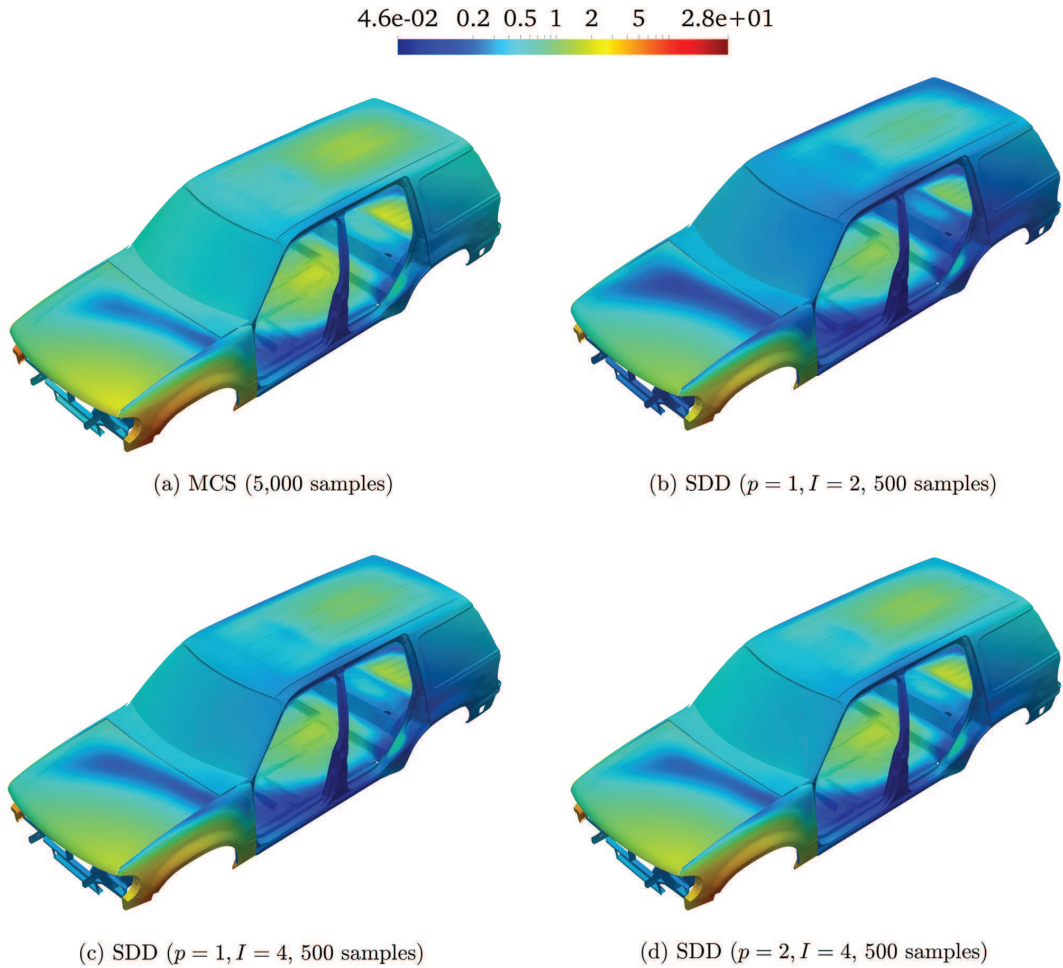}
\par\end{centering}
\caption{
Contour plots of standard deviations of the 15th mode shape of the BIW model in logarithmic scale by various methods;
(a) crude MCS (5000 samples);
(b) univariate SDD ($p=1$, $I=2$);
(c) univariate SDD ($p=1$, $I=4$);
(d) univariate SDD ($p=2$, $I=4$).
}
\label{fig8}
\end{figure}

\section{Concluding remarks}
A new dimensional decomposition, referred to as SDD, is introduced for high-dimensional uncertainty quantification analysis of complex systems, including those featuring nonsmooth functions, subject to independent but otherwise arbitrary probability measures of input random variables. From the decomposition, a square-integrable output random variable is expanded in terms of hierarchically ordered measure-consistent multivariate orthonormal B-splines in input random variables.  A dimensionwise splitting of appropriate spline spaces into orthogonal subspaces, each spanned by a reduced set of measure-consistent orthonormal B-splines, was configured, resulting in SDD of a general $L^2$-function of input variables. Under the prescribed assumptions, the set of measure-consistent orthonormal B-splines forms a basis of each subspace, leading to an orthogonal sum of such sets of basis functions, including the constant subspace, to span the space of all splines. The approximation quality of the expansion was demonstrated in terms of the modulus of smoothness of the function, resulting in the mean-square convergence of SDD to the correct limit. The weaker modes of convergence, such as those in probability and in distribution, transpire naturally. The optimality of SDD, including deriving a few existing expansions as special cases of SDD, was affirmed.  A truncated SDD, obtained by retaining terms associated with a chosen degree of interaction, contains products of univariate splines determined by the same degree of interaction.  Therefore, the computational complexity of a truncated SDD is polynomial, as opposed to exponential, thus deflating the curse of dimensionality to the magnitude possible.  Analytical formulae are proposed to calculate the mean and variance of an SDD approximation for a general output random variable in terms of the expansion coefficients.

Numerical results entailing nonsmooth functions indicate that a low-order SDD approximation equipped with an adequate mesh size is capable of estimating the second-moment properties or probability distributions that are as or more accurate than those obtained from a high-order PCE, PDD, and sparse-grid approximations. Finally, the success of conducting UQ analysis for a 34-dimensional SUV BIW system demonstrates the practicality of the SDD method in solving industrial-scale engineering problems.

\appendix
\section{Knot sequence}
Let $\mathbf{x}=(x_1,\ldots,x_N)$ be an arbitrary point in $\mathbb{A}^N$.  For the coordinate direction $k$, $k=1,\ldots,N$, define a non-negative integer $p_k \in \mathbb{N}_0$ and a positive integer $n_k \ge (p_k+1) \in \mathbb{N}$, representing the degree and total number of basis functions, respectively. In order to define B-splines, the concept of knot sequence, also referred to as knot vector, for each coordinate direction $k$ is needed.

A knot sequence $\boldsymbol{\xi}_k$ for the interval $[a_k,b_k] \subset \mathbb{R}$, given $n_k > p_k \ge 0$, is a vector comprising non-decreasing sequence of real numbers
\begin{equation}
\begin{array}{c}
\boldsymbol{\xi}_k:={\{\xi_{k,i_k}\}}_{i_k=1}^{n_k+p_k+1}=
\{a_k=\xi_{k,1},\xi_{k,2},\ldots,\xi_{k,n_k+p_k+1}=b_k\},   \\
\xi_{k,1} \le \xi_{k,2} \le \cdots \le \xi_{k,n_k+p_k+1}, \rule{0pt}{0.2in}
\end{array}
\label{A.1}
\end{equation}
where $\xi_{k,i_k}$ is the $i_k$th knot with $i_k=1,2,\ldots,n_k+p_k+1$ representing the knot index for the coordinate direction $k$. The elements of $\boldsymbol{\xi}_k$ are called knots.

According to \eqref{A.1}, there are a total of $n_k+p_k+1$ knots, which may be equally or unequally spaced. To monitor knots without repetitions, denote by $\zeta_{k,1},\ldots,\zeta_{k,r_k}$ the $r_k$ distinct knots in $\boldsymbol{\xi}_k$ with respective multiplicities $m_{k,1},\ldots,m_{k,r_k}$. Then the knot sequence in \eqref{A.1} can also be expressed by
\begin{equation}
\begin{array}{c}
\boldsymbol{\xi}_k=\{a_k=\overset{m_{k,1}~\mathrm{times}}{\overbrace{\zeta_{k,1},\ldots,\zeta_{k,1}}},
\overset{m_{k,2}~\mathrm{times}}{\overbrace{\zeta_{k,2},\ldots,\zeta_{k,2}}},\ldots,
\overset{m_{k,r_{k}-1}~\mathrm{times}}{\overbrace{\zeta_{k,r_{k}-1},\ldots,\zeta_{k,r_{k}-1}}},
\overset{m_{k,r_{k}}~\mathrm{times}}{\overbrace{\zeta_{k,r_{k}},\ldots,\zeta_{k,r_{k}}}}=b_k\}, \\
a_k=\zeta_{k,1} < \zeta_{k,2} < \cdots < \zeta_{k,r_k-1} < \zeta_{k,r_k}=b_k,  \rule{0pt}{0.2in}
\end{array}
\label{A.2}
\end{equation}
which consists of a total number of
\[
\sum_{j_k=1}^{r_k} m_{k,j_k} = n_k+p_k+1
\]
knots. As shown in \eqref{A.2}, each knot, whether interior or exterior, may appear $1 \le m_{k,j_k} \le p_k+1$ times, where $m_{k,j_k}$ is referred to as its multiplicity. The multiplicity has important implications on the regularity properties of B-spline functions. A knot sequence is called open if the end knots have multiplicities $p_k+1$. In this case, definitions of more specific knot sequences are in order.

A knot sequence is said to be $(p_k+1)$-open if the first and last knots appear $p_k+1$ times, that is, if
\begin{equation}
\begin{array}{c}
\boldsymbol{\xi}_k=\{a_k=\overset{p_k+1~\mathrm{times}}{\overbrace{\zeta_{k,1},\ldots,\zeta_{k,1}}},
\overset{m_{k,2}~\mathrm{times}}{\overbrace{\zeta_{k,2},\ldots,\zeta_{k,2}}},\ldots,
\overset{m_{k,r_{k}-1}~\mathrm{times}}{\overbrace{\zeta_{k,r_{k}-1},\ldots,\zeta_{k,r_{k}-1}}},
\overset{p_k+1~\mathrm{times}}{\overbrace{\zeta_{k,r_{k}},\ldots,\zeta_{k,r_{k}}}}=b_k\}, \\
a_k=\zeta_{k,1} < \zeta_{k,2} < \cdots < \zeta_{k,r_k-1} < \zeta_{k,r_k}=b_k. \rule{0pt}{0.2in}
\end{array}
\label{A.3}
\end{equation}

A knot sequence is said to be $(p_k+1)$-open with simple knots if it is $(p_k+1)$-open and all interior knots appear only once, that is, if
\begin{equation}
\begin{array}{c}
\boldsymbol{\xi}_k=\{a_k=\overset{p_k+1~\mathrm{times}}{\overbrace{\zeta_{k,1},\ldots,\zeta_{k,1}}},
\zeta_{k,2},\ldots,
\zeta_{k,r_{k}-1},
\overset{p_k+1~\mathrm{times}}{\overbrace{\zeta_{k,r_{k}},\ldots,\zeta_{k,r_{k}}}}=b_k\}, \\
a_k=\zeta_{k,1} < \zeta_{k,2} < \cdots < \zeta_{k,r_k-1} < \zeta_{k,r_k}=b_k. \rule{0pt}{0.2in}
\end{array}
\label{A.5}
\end{equation}

A $(p_k+1)$-open knot sequence with or without simple knots is commonly found in applications \cite{cottrell09}.

\subsection*{Acknowledgments}
The authors thank the anonymous reviewers and the associate editor for providing suggestions to improve the previous version of the paper.

\bibliographystyle{siamplain}
\bibliography{sdd}

\begin{thebibliography}{10}

\bibitem{bellman57}
{\sc R.~Bellman}, {\em Dynamic Programming}, Princeton University Press:
  Princeton, NJ, 1957.

\bibitem{cameron47}
{\sc R.~H. Cameron and W.~T. Martin}, {\em The orthogonal development of
  non-linear functionals in series of {Fourier}-{Hermite} functionals}, Ann.
  Math., 48 (1947), pp.~385--392.

\bibitem{cottrell09}
{\sc J.~A. Cottrell, T.~J.~R. Hughes, and Y.~Bazilevs}, {\em Isogeometric
  Analysis: Toward Integration of CAD and FEA}, John Wiley \& Sons, 2009.

\bibitem{cox72}
{\sc M.~G. Cox}, {\em The numerical evaluation of {B}-splines}, Journal of
  Institute of Mathematics and its Applications, 10 (1972), pp.~134--149.

\bibitem{abaqus2019}
{\sc Dassault Systemes Simulia Corp.}, {\em Abaqus Standard, Version 2019},
  2019.

\bibitem{deboor72}
{\sc C.~De~Boor}, {\em On calculation with {B}-splines}, Journal of
  Approximation Theory, 6 (1972), pp.~50--62.

\bibitem{dixler21}
{\sc S.~Dixler, R.~Jahanbin, and S.~Rahman}, {\em Uncertainty quantification by
  optimal spline dimensional decomposition}, International Journal for
  Numerical Methods in Engineering, 122 (2021), pp.~5898--5934.

\bibitem{efron81}
{\sc B.~Efron and C.~Stein}, {\em The jackknife estimate of variance}, The
  Annals of Statistics, 9 (1981), pp.~pp. 586--596.

\bibitem{foo08}
{\sc J.~Foo, X.~Wan, and G.~Karniadakis}, {\em The multi-element probabilistic
  collocation method {(ME-PCM)}: error analysis and applications}, Journal of
  Computational Physics, 227 (2008), pp.~9572--9595.

\bibitem{gerstner98}
{\sc T.~Gerstner and M.~Griebel}, {\em Numerical integration using sparse
  grids}, Numerical Algorithms, 18 (1998), pp.~209--232.

\bibitem{golub96}
{\sc G.~H. Golub and C.~F. van Loan}, {\em Matrix computations}, The John
  Hopkins University Press, third~ed., 1996.

\bibitem{griebel10}
{\sc M.~Griebel and M.~Holtz}, {\em Dimension-wise integration of
  high-dimensional functions with applications to finance}, J. Complex., 26
  (2010), pp.~455--489.

\bibitem{hoeffding48}
{\sc W.~Hoeffding}, {\em A class of statistics with asymptotically normal
  distribution}, The Annals of Mathematical Statistics, 19 (1948), pp.~pp.
  293--325.

\bibitem{jahanbin20}
{\sc R.~Jahanbin and S.~Rahman}, {\em Stochastic isogeometric analysis in
  linear elasticity}, Computer Methods in Applied Mechanics and Engineering,
  364 (2020), pp.~1--38.
\newblock Article 112928.

\bibitem{kessy16}
{\sc A.~Kessy, A.~Lewin, and K.~Strimmer}, {\em Optimal whitening and
  decorrelation}, The American Statistician, DOI: 10.1080/00031305.2016.1277159
  (2018).

\bibitem{kuo10}
{\sc F.~Y. Kuo, I.~H. Sloan, G.~W. Wasilkowski, and H.~Wozniakowski}, {\em On
  decompositions of multivariate functions}, Mathematics of Computation, 79
  (2011), pp.~953--966.

\bibitem{lanczos50}
{\sc C.~Lanczos}, {\em An iteration method for the solution of the eigenvalue
  problem of linear differential and integral operators}, Journal of Research
  of the National Bureau of Standard, 45 (1950), pp.~255--282.

\bibitem{owen97}
{\sc A.~B. Owen}, {\em Monte {Carlo} variance of scrambled net quadrature},
  SIAM Journal on Numerical Analysis, 34 (1997), pp.~pp. 1884--1910.

\bibitem{piegl97}
{\sc L.~A. Piegl and W.~Tiller}, {\em The NURBS Book, Second Edition},
  Springer-Verlag: Berlin, 1997.

\bibitem{rabitz99}
{\sc H.~Rabitz and O.~Alis}, {\em General foundations of high dimensional model
  representations}, Journal of Mathematical Chemistry, 25 (1999), pp.~197--233.
\newblock 10.1023/A:1019188517934.

\bibitem{rahman14}
{\sc S.~Rahman}, {\em Approximation errors in truncated dimensional
  decompositions}, Mathematics of Computation, 83 (2014), pp.~2799--2819.

\bibitem{rahman14b}
{\sc S.~Rahman}, {\em A generalized {ANOVA} dimensional decomposition for
  dependent probability measures}, SIAM/ASA Journal on Uncertainty
  Quantification, 2 (2014), pp.~670--697.

\bibitem{rahman18}
{\sc S.~Rahman}, {\em Mathematical properties of polynomial dimensional
  decomposition}, SIAM/ASA Journal on Uncertainty Quantification, 6 (2018),
  pp.~816--844.

\bibitem{rahman18b}
{\sc S.~Rahman}, {\em A polynomial chaos expansion in dependent random
  variables}, Journal of Applied Mathematics and Applications, 4 (2018),
  pp.~1--26.

\bibitem{rahman20}
{\sc S.~Rahman}, {\em A spline chaos expansion}, SIAM/ASA Journal on
  Uncertainty Quantification, 8 (2020), pp.~27--57.

\bibitem{schumaker07}
{\sc L.~J. Schumaker}, {\em Spline Functions: Basic Theory}, Cambridge
  University Press: Cambridge, third~ed., 2007.

\bibitem{smith13}
{\sc R.~Smith}, {\em Uncertainty Quantification: Theory, Implementation, and
  Applications}, SIAM: New York, 2013.

\bibitem{smolyak63}
{\sc S.~Smolyak}, {\em Quadrature and interpolation formulas for tensor
  products of certain classes of functions}, Dokl. Akad. Nauk SSSR, 4 (1963),
  pp.~240--243.

\bibitem{sobol03}
{\sc I.~M. Sobol}, {\em Theorems and examples on high dimensional model
  representation}, Reliability Engineering \& System Safety, 79 (2003), pp.~187
  -- 193.

\bibitem{sullivan15}
{\sc T.~J. Sullivan}, {\em Introduction to Uncertainty Quantification},
  Springer: New York, 2015.

\bibitem{wiener38}
{\sc N.~Wiener}, {\em The homogeneous chaos}, American Journal of Mathematics,
  60 (1938), pp.~897--936.

\bibitem{yadav13}
{\sc V.~Yadav and S.~Rahman}, {\em Uncertainty quantification of
  high-dimensional complex systems by multiplicative polynomial dimensional
  decompositions}, International Journal for Numerical Methods in Engineering,
  94 (2013), pp.~221--247.

\end{thebibliography}

\end{document}